\documentclass[12pt]{article}

\usepackage{epsfig,amsmath,amssymb,graphics,verbatim,amsfonts,subfigure,psfrag,amsthm}
\usepackage[letterpaper,margin=0.9in]{geometry}
\usepackage{tikz}
\usepackage[percent]{overpic}

\usepackage{enumitem}
\makeatletter
\newcommand{\mylabel}[2]{#2\def\@currentlabel{#2}\label{#1}}
\makeatother

\let\oldbibliography\thebibliography
\renewcommand{\thebibliography}[1]{%
  \oldbibliography{#1}%
  \setlength{\itemsep}{0.5mm}%
}

\newtheorem{thm}{Theorem}[section]

\newtheorem{lem}[thm]{Lemma}

\newtheorem{defn}[thm]{Definition}

\def\R{\mathbb{R}}

\def\N{\mathbb{N}}

\def\I{\infty}

\def\txtd{{\textnormal{d}}}
\def\txte{{\textnormal{e}}}
\def\txth{{\textnormal{h}}}

\def\txtD{{\textnormal{D}}}

\newcommand{\be}{\begin{equation}}
\newcommand{\ee}{\end{equation}}
\newcommand{\bea}{\begin{eqnarray}}
\newcommand{\eea}{\end{eqnarray}}
\newcommand{\beann}{\begin{eqnarray*}}
\newcommand{\eeann}{\end{eqnarray*}}
\newcommand{\benn}{\begin{equation*}}
\newcommand{\eenn}{\end{equation*}}

\DeclareMathSymbol{\leqsymb}{\mathalpha}{AMSa}{"36}

\DeclareMathSymbol{\geqsymb}{\mathalpha}{AMSa}{"3E}

\def\ra{\rightarrow}
\def\I{\infty}

\newcommand{\cA}{{\mathcal A}}  
\newcommand{\cC}{{\mathcal C}}  
\newcommand{\cG}{{\mathcal G}}  
\newcommand{\cK}{{\mathcal K}}  
\newcommand{\cO}{{\mathcal O}}  
\newcommand{\cS}{{\mathcal S}}  

\begin{document}

\title{Generalized Play Hysteresis Operators\\ in Limits of Fast-Slow Systems}

\author{Christian Kuehn\footnotemark[1]~
and Christian M\"unch\footnotemark[2]}
\date{\today}

\maketitle

\renewcommand{\thefootnote}{\fnsymbol{footnote}}

\footnotetext[1]{Technical University of Munich, Faculty of Mathematics, Research Unit 
``Multiscale and Stochastic Dynamics'', 85748 Garching b.~M\"unchen, Germany}
\footnotetext[2]{Technical University of Munich, Faculty of Mathematics, Research Unit 
``Mathematical Modelling'', 85748 Garching b.~M\"unchen, Germany}

\begin{abstract}
Hysteresis operators appear in many applications such as elasto-plasticity
and micromagnetics, and can be used for a wider class of systems, where rate-independent 
memory plays a role. A natural approximation for systems of evolution equations with hysteresis 
operators are fast-slow
dynamical systems, which - in their used approximation form - do not involve any memory
effects. Hence, viewing differential equations with hysteresis operators in the non-linearity 
as a limit of approximating fast-slow dynamics
involves subtle limit procedures. In this paper, we give a proof of Netushil's ``observation'' 
that broad classes of planar fast-slow systems with a two-dimensional critical manifold 
are expected to yield generalized play operators in the singular limit. We 
provide two proofs of this ``observation'' based upon the fast-slow systems paradigm
of decomposition into subsystems. One proof strategy employs suitable convergence 
in function spaces, while the second approach considers a geometric
strategy via local linearization and patching adapted originally from problems in stochastic
analysis. We also provide an illustration of our results
in the context of oscillations in forced planar non-autonomous fast-slow systems. The study
of this example also strongly suggests that new canard-type mechanisms can occur for 
two-dimensional critical manifolds in planar systems.
\end{abstract}


\section{Introduction}
\label{sec:intro}

In this (non-technical) introduction, we are going to outline the main topic. 
We also present the main result and proof strategy used in this paper.\medskip 

Depending upon the context, the term 'hysteresis' is used in technically 
different, yet strongly related, forms. Classical results in magnetic 
materials~\cite{StonerWohlfarth,DellaTorre} and mechanical 
systems~\cite{Love,MuellerXu} refer to hysteresis as the description of memory effects, 
particularly with a focus on rate-independent memory~\cite{MielkeTheil}; see 
Section~\ref{ssec:hysteresis} for the relevant mathematical definitions for this work
and~\cite{Mayergoyz,BrokateSprekels,KrasnoselskiiPokrovskii,Visintin,MielkeRoubicek} for 
further background literature. One observation for several classes of dynamical 
systems defined by \emph{hysteresis operators} is that trajectories often form 
'loops' in phase space as shown in Figure~\ref{fig:01}(a). A second common use of 
hysteresis is to describe a system exhibiting switching between two locally stable 
states upon parameter variation. A typical situation occurs when two fold 
bifurcations~\cite{GH} are connected in an S-shaped curve as shown in 
Figure~\ref{fig:01}(b). This \emph{hysteresis effect} has been found and described 
in essentially all disciplines involving mathematical models ranging from 
neuroscience~\cite{FitzHugh}, geoscience~\cite{GanopolskiRahmstorf}, 
engineering~\cite{vanderPol1}, solid-state physics~\cite{Strogatz}, 
ecology~\cite{BeisnerHaydonCuddington} to economics~\cite{Cross}, 
just to name a few; see also the references in~\cite{KuehnBook}. A suitable 
mathematical formulation to approximate systems, in which this hysteresis effect appears, 
are fast-slow ordinary differential equations (ODEs), where parameters are viewed as 
slowly-driven variables. Also in this case, limit cycles are frequently 
observed, e.g., slowly moving through an S-shaped bifurcation diagram 
can lead to relaxation oscillations; see also Section~\ref{ssec:fastslow} 
for more precise definitions and background on geometric singular 
perturbation theory (GSPT) for fast-slow systems~\cite{KuehnBook}.\medskip 

\begin{figure}[htbp]
	\centering
\begin{overpic}[width=0.75\linewidth]{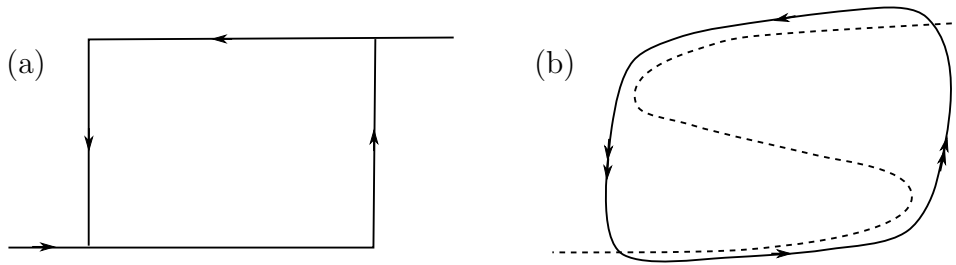}
\put(0,20){(a)}
\put(55,20){(b)}
\end{overpic}
\caption{\label{fig:01}Sketch of two hysteresis models. (a) Trajectory for 
		a \emph{hysteresis operator} with memory; note that if we start at the lower left point, 
		then the trajectory does not make sense as a trajectory of a planar ODE as it would violate 
		local uniqueness. (b) \emph{Hysteresis behaviour} given by a relaxation oscillation (solid 
		curve) in a fast-slow system. The S-shaped critical manifold (dashed curve) can also be 
		interpreted as one limiting part of the dynamics in the infinite-time scale separation 
		limit, i.e., there are three branches of steady states, which meet at two fold bifurcations.}
\end{figure}

Although these two uses of \emph{hysteresis} differ in their mathematical setup, it is 
evident from Figure~\ref{fig:01} that one might expect at least some relation between 
the model classes for certain cases. The ``observation'' that certain hysteresis operators
can be obtained as a singular limit from fast-slow dynamics is often 
attributed~\cite{PokrovskiiSobolev,MortellOMalleyPokrovskiiSobolev} 
to Netushil~\cite{Netushil1,Netushil2}.   
However, taking limits directly to relate fast-slow dynamics to systems with 
hysteresis operators is difficult as the limits are singular, in the sense that we have to bridge via 
the limiting process two different classes of equations. This difficulty is one 
possible motivation to look at the singular limit by exploiting gradient-type 
structures~\cite{MielkeRoubicek}, if they are available, or going deeper into the
theory of differential inclusions. Albeit yielding limits for certain subclasses of
fast-slow systems, these approaches tend to make very strong assumptions on the 
structure of the system and, more importantly, the interpretation and proof strategies
provided by GSPT, i.e., guided by the individual decomposition of trajectories,
are lost. Hence, we are going to refer to this situation more precisely as 
Netushil's conjecture, i.e., how one can \emph{rigorously unify GSPT and the abstract 
hysteresis operator viewpoint}, without making too stringent assumptions on the 
structure of the fast-slow system; see Sections 3 \& 4 for the precise statements.\medskip 

In this work, we give two rigorous proofs of Netushil's conjecture for a very general class
of planar fast-slow systems, which have \emph{two-dimensional} critical 
manifolds~\cite{PokrovskiiSobolev}; see also Section~\ref{sec:main1} 
for the technical statement of Netushil's conjecture. We highlight that in our setting,
we allow for a full coupling of fast and slow variables. We
establish our results via two different techniques. The first approach is more 
functional-analytic, yet carefully exploits the fast-slow decomposition. The second
approach is very geometric and transfers a proof-strategy initially developed for
stochastic fast-slow systems to the deterministic hysteresis situation. More precisely, 
we show how to obtain the generalized play operator, which can be viewed as a model 
for perfect plasticity. In addition, to the overall GSPT approach, each of the two 
proofs also contains further technical advances, e.g., the use of suitable projectors
to deal with small-scale oscillations or a careful matching argument near two-dimensional
critical manifolds in planar systems. In summary, our results aim to provide a 
better bridge between two adjacent areas of mathematical modelling, and to provide a 
step in establishing Netushil's conjecture, i.e., to link GSPT and hysteresis operators 
in even more generality; see Section~\ref{ssec:hysteresis} for background and 
comparison of our results to other approaches.\medskip

The paper is structured as follows. In Section~\ref{sec:background}, we explain 
the necessary background from fast-slow systems and hysteresis operators. We also 
provide a brief introduction to motivate our approach and compare it to other possible 
approaches not using GSPT and fast-slow decomposition. In Section~\ref{sec:main1}, 
we state the first version of our main result and provide a more functional-analytic 
decomposition proof. In Section~\ref{sec:main2}, we state a variant of the 
convergence result, which provides stronger convergence under stronger 
assumptions. The proof of this result follows a different strategy patching 
fast dynamics geometrically. In Section~\ref{sec:appl}, we consider an 
application of our results to a prototypical periodically forced planar fast-slow 
system, which displays interesting oscillatory patterns. This example also illustrates
the possible occurrence and the role of canard trajectories.\medskip

\textbf{Remark on notation:} If not further specified, we always write 
$\|\cdot\|=\|\cdot\|_2$ for the Euclidean norm. The same remark applies to matrix
norms and metrics on finite-dimensional spaces, they will all be understood in the 
Euclidean sense unless specified otherwise. 

\section{Background}
\label{sec:background}

In this section we provide some necessary background for the two main areas in this paper. 
We hope that this is going to make the results more accessible for experts in different 
fields. We restrict ourselves here to the case of planar systems to simplify the notation 
and present the main ideas. The techniques we present are expected to generalize to 
several classes of higher-dimensional systems.
  
\subsection{Fast-Slow Systems}
\label{ssec:fastslow}

Let $(x,y)\in\R^2$ and consider the planar fast-slow system
\be
\label{eq:fs}
\begin{array}{rcrcl}
\varepsilon \frac{\txtd x}{\txtd t} &=& \varepsilon \dot{x} &=& f(x,y;\varepsilon),\\
\frac{\txtd y}{\txtd t} &=& \dot{y} &=& g(x,y;\varepsilon),
\end{array}
\ee
where $\varepsilon>0$ is a small parameter indicating the time-scale separation 
between the fast variable $x$ and the slow variable $y$. For now we shall assume 
that $f:\R^3\ra \R$ and $g:\R^3\ra \R$ are sufficiently smooth but for our main 
setting we shall considerably weaken this assumption and also allow the slow vector
field to be non-autonomous. A classical goal in fast-slow systems is to understand 
the dynamics of~\eqref{eq:fs} for sufficiently small $\varepsilon$ by 
showing persistence results from the singular limit $\varepsilon\ra 0$. 
Setting $\varepsilon=0$ in~\eqref{eq:fs} yields the slow subsystem
\be
\label{eq:slowsub}
\begin{array}{rcl}
0 &=& f(x,y;0),\\
\dot{y} &=& g(x,y;0),
\end{array}
\ee
which is a differential-algebraic equation on the critical set 
\be
\cC_0:=\{(x,y)\in\R^2:f(x,y;0)=0\}.
\ee
Therefore, the slow subsystem~\eqref{eq:slowsub} only covers the dynamics in a part 
of phase space. Another possibility is to consider the fast-slow system~\eqref{eq:fs} 
on the fast time scale $s:=\varepsilon t$ instead of the slow time scale $t$. 
Taking the singular limit on the fast time scale leads to the fast subsystem  
\be
\label{eq:fastsub}
\begin{array}{rcrcl}
\frac{\txtd x}{\txtd s} &=& x' &=& f(x,y;0),\\
\frac{\txtd y}{\txtd s} &=& y' &=& 0,\\
\end{array}
\ee
which is a parametrized differential equation, where the slow variable $y$ acts as 
a parameter. The fast and slow subsystems are different types of differential 
equations, which illustrates the singular perturbation character of fast-slow systems. 
For~\eqref{eq:fastsub} the critical set $\cC_0$ consists of the equilibrium points (or 
steady states), while it is a constraint for~\eqref{eq:slowsub}.\medskip

Several approaches to analyze fast-slow systems exist~\cite{KuehnBook}. Probably
the most classical technique is to use asymptotic methods, such as matched 
asymptotic expansions~\cite{MisRoz,KevorkianCole,BenderOrszag}. In this approach,
we aim to write the solution as
\benn
(x(s),y(s))\sim \left(\sum_{k=1}^Kx_k(s)a_k(\varepsilon),\sum_{k=1}^Ky_k(s)
b_k(\varepsilon)\right)+\cO(b_{k+1}(\varepsilon),a_{k+1}(\varepsilon)), 
\qquad \text{as $\varepsilon\ra 0$,}
\eenn  
for some $K\in\N$, where $\{a_k(\varepsilon)\}_{k=1}^\I,\{b_k(\varepsilon)\}_{k=1}^\I$
are asymptotic sequences; the simplest example are polynomials 
$a_k(\varepsilon)=\varepsilon^k=b_k(\varepsilon)$. In asymptotic analysis, one
usually has to \emph{match} different series by combining solutions obtained of
a slow time scale from~\eqref{eq:slowsub} and via a fast time scale from~\eqref{eq:fastsub}.
In particular, the crucial point is to exploit the decomposition algebraically.\medskip

More recently, the development of geometric singular perturbation theory (GSPT)
has provided very significant additional \emph{geometric insight}, how we may
combine~\eqref{eq:slowsub}-\eqref{eq:fastsub}. In this work, we restrict to the 
case when the critical set $\cC_0$ is a manifold but 
for other cases see~\cite{KruSzm4,Schecter,KuehnSzmolyan}. Consider a compact 
submanifold $\cS_0\subset \cC_0$. Then $\cS_0$ is called normally 
hyperbolic~\cite{Fenichel4,Jones,Kaper} if $\partial_x f(p;0)\neq 0$ for all $p\in\cS_0$; 
for a higher-dimensional fast variable, normal hyperbolicity is defined by requiring that 
the matrix $\txtD_x f(p;0)$ is a hyperbolic matrix. Fenichel's 
Theorem~\cite{Fenichel4,Jones,WigginsIM} guarantees that $\cS_0$ perturbs to a locally 
invariant slow manifold $\cS_\varepsilon$, which is $\cO(\varepsilon)$-close 
to $\cS_0$. The dynamics on $\cS_\varepsilon$ is conjugate to 
the dynamics on $\cS_0$, which converts the singular perturbation problem to a regular 
perturbation problem near a normally hyperbolic critical manifold. In this view,
asymptotic matching becomes geometric matching of singular limit trajectories from 
the fast and slow subsystems; e.g., in Figure~\ref{fig:01}(b), we have to match two fast
segments with two slow segments near the two fold points, where the critical 
manifold loses normal hyperbolicity~\cite{KuehnBook}.\medskip

In summary, the crucial point of this discussion for the current work is, that
modern GSPT, as well as more algebraic asymptotic matching, crucially exploit the views
of \emph{decomposition}, \emph{scaled subsystems}, and \emph{matching regions/points}. 
There is a very detailed literature available on many further important topics in 
fast-slow systems utilizing this strategy and we refer to the literature review 
in~\cite{KuehnBook} for a broader view of the area.\medskip

\begin{figure}[htbp]
	\centering
\begin{overpic}[width=0.8\linewidth]{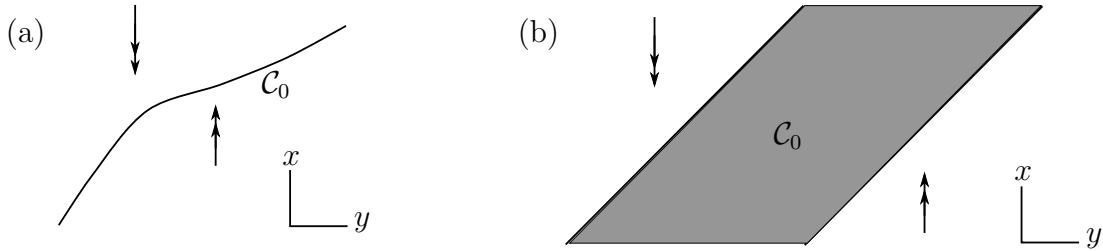}
\put(-5,20){(a)}
\put(45,20){(b)}
\put(29,2){$y$}
\put(22,8){$x$}
\put(100.5,0.5){$y$}
\put(93.5,6.5){$x$}
\put(70,10){$\cC_0$}
\put(20,15){$\cC_0$}
\end{overpic}
\caption{\label{fig:02}Sketch of two possible configurations of the critical manifold
		$\cC_0$ for planar fast-slow systems. (a) Classical situation with a dimension one (and
		codimension one) critical manifold. (b) Situation discussed in this paper, when $\cC_0$ 
		has the same dimension as the ambient phase space, i.e., dimension two and codimension 
		zero; see also equation~\eqref{eq:exdegen}.}
\end{figure}

One crucial aspect of the critical manifold in planar systems is that it is generically 
of dimension one (or even the empty set) if it is defined via the zeros of a generic smooth 
mapping $f$. However, 
the disadvantage is that we cannot model two-dimensional constrained dynamics in this context 
in planar fast-slow systems. For example, consider the fast variable vector field
\be
\label{eq:exdegen}
f(x,y;\varepsilon):=
\left\{\begin{array}{ll}
-x + y & \text{ if $y>x$,}\\
0 & \text{ if $x-1<y<x$,}\\
-x + y + 1 & \text{ if $y<x-1$.}\\
\end{array}\right.
\ee     
The critical manifold $\cC_0=\{(x,y)\in\R^2:x-1<y<x\}$ is a two-dimensional strip. The 
classical normal hyperbolicity assumption does not hold for $\cC_0$ since the fast vector 
field is identically zero inside the strip. In particular, one may ask, what happens in 
the singular limit $\varepsilon\ra 0$ to the flow of the fast-slow system~\eqref{eq:fs} 
with fast vector field similar to~\eqref{eq:exdegen}. Netushil's 
observation~\cite[Sec.~1.1.2]{PokrovskiiSobolev} is that hysteresis operators 
play a key role to capture the singular limit in this situation.\medskip

One major obstacle to carry out GSPT directly using Fenichel's Theorem
and related results is that one has to solve the problem in another
limiting direction, which is not customary for the GSPT approach. More precisely, 
the limit $\varepsilon= 0$ is unknown, whereas it is usually \emph{assumed} in GSPT
that it is easier to analyze the fast and slow subsystems, i.e., usually 
the difficult part is to control the case $0<\varepsilon\ll 1$. In the context
of Netushil's observation, the geometry is significantly worse as we only have
$\partial \cC_0$ as the main geometric object available with just one-sided 
estimates near $\partial \cC_0$ for the fast dynamics. Additionally, observe that 
we work in the coupled setting of a fast-slow system rather than only with one 
equation with an input function as in the classical statement for Netushil's 
observation. Here we overcome these 
difficulties using a suitable local decomposition of trajectories as well as 
projections to preserve the geometric viewpoint.\medskip

Other approaches to the problem use very different technical approaches, and we
briefly review these approaches in the next section, which also contains some 
basic background on hysteresis operators.

\subsection{Hysteresis Operators}
\label{ssec:hysteresis}

Amongst others, hysteresis effects appear in physics in fields like ferromagnetism, 
ferroelectricity or 
plasticity~\cite{Mayergoyz,BrokateSprekels,KrasnoselskiiPokrovskii,Visintin,MielkeRoubicek}. 
They can also be observed in shape memory effects of certain materials, they are relevant for 
thermostats in engineering~\cite{Visintin}, and are used for modelling of certain systems
in mathematical biology~\cite{GurevichShaminTikhomirov,HoppensteadtJaeger,Kopfova,
Pimenovetal,CurranGurevichTikhomirov}. Mathematically, these effects can be described 
by hysteresis operators such as the scalar play~\cite{BrokateKrejci}, the scalar 
stop~\cite{BrokateRachinskii} or the Prandtl-Ishlinski\v{\i} operator~\cite{Kuhnen}, 
to only name a few.\medskip 

We focus on scalar hysteresis operators in this paper. Given a time interval $[0,T]$, 
scalar hysteresis operators take an admissible time-dependent input 
function $y:[0,T]\rightarrow \R$ together with an initial value $x_0$ and return 
a time-dependent output function $x=x(y,x_0)$, where we can also view $x$ as a 
map $x:[0,T]\rightarrow\R $ if $x_0$ is fixed. All scalar rate-independent hysteresis 
operators have two properties in common~\cite{Visintin,BrokateSprekels}:
\begin{defn}
\leavevmode
\begin{itemize}
	\item[(Vol)] The output function $x(t)$ at time $t\in [0,T]$ may depend not only on the value of 
	the input function $y(t)$ at time $t$, but on the whole history of $y$ in the interval 
	$[0,t]$. This non-locality in time is often referred to as memory effect, causality or 
	Volterra property: for all $y_1,y_2$ in the domain of the operator, for all initial values
	$x_0$, and any $t\in [0,T]$ it follows that if $y_1=y_2$ in $[0,t]$, then 
	$[x(y_1,x_0)](t)=[x(y_2,x_0)](t)$; cf.~\cite[Chapter~III]{Visintin}. 
	\item[(RI)] The output function $x$ is invariant under time transformations. This means that 
	for any monotone increasing and continuous function $\phi:[0,T]\rightarrow [0,T]$ with 
	$\phi(0)=0$ and $\phi(T)=T$ and for all admissible input functions $y$ it holds
	\benn
	[x(y\circ\phi,x_0)](t)=x(y,x_0)(\phi(t)),\qquad \forall t\in[0,T].
	\eenn
	In~\cite[Chapter III]{Visintin}, the function $\phi$ is also assumed to be bijective,
	i.e., the definitions differ in the literature. For our purpose one may consider either
	definition of admissible time transformations. Invariance under time transformations 
	is also called rate-independence in the literature \cite[Definition 1.2.1]{MielkeRoubicek}.
\end{itemize}
\end{defn}

Furthermore, we recall that hysteresis operators may not be described by planar differential 
equations. To see this, consider an ODE of the form
\benn
\frac{\txtd}{\txtd t}x_i(t)=f(x_i(t),y_i(t)),\ x_i(0)= x_{0,i},\ i\in\{1,2\}
\eenn
for two different input functions $y_1,y_2$ and initial values $x_{0,1},x_{0,2}$. If for a 
given time $t\in [0,T]$ it holds $x_1(t)=x_2(t)$ and if $y_1|_{(t,T]}=y_2|_{(t,T]}$, then 
clearly $x_1|_{[t,T]}=x_2|_{[t,T]}$ no matter if $y_1|_{[0,t]}\neq y_2|_{[0,t]}$ or even 
$x_{0,1}\neq x_{0,2}$. Therefore, the Volterra property cannot be captured in general. The solutions 
$x_1,x_2$ of the ODEs on the remaining time interval $(t,T]$ depend only on the current 
values $x_1(t)$ and $x_2(t)$ at time $t$ and on the behaviour of the input functions 
$y_1$ and $y_2$ in the interval $(t,T]$; see also Figure~\ref{fig:01}(a). Differential equations 
like
\benn
\frac{\txtd}{\txtd t}x(t)=f(x(t),y(t)),\qquad x(0)= x_0,
\eenn
are also not rate-independent. To see this, consider a time transformation $\phi$. Then for 
the solution operator $x=x(y,x_0)$ of the differential equation it holds
\benn
x(y\circ\phi,x_0)(t)=x_0 + \int_{0}^{t} f(x(s),y(\phi(s)))~ \txtd s,
\eenn
whereas we find that
\benn
[x(y,x_0)\circ\phi](t)=x_0 + \int_{0}^{\phi(t)} f(x(s),y(s))~\txtd s
\eenn
for $t\in [0,T]$. In general those two functions do not coincide. In summary, the appearance 
of differential equations with hysteresis as the singular limit of systems of differential 
equations brings along completely new features in the evolution dynamics.\medskip
	
	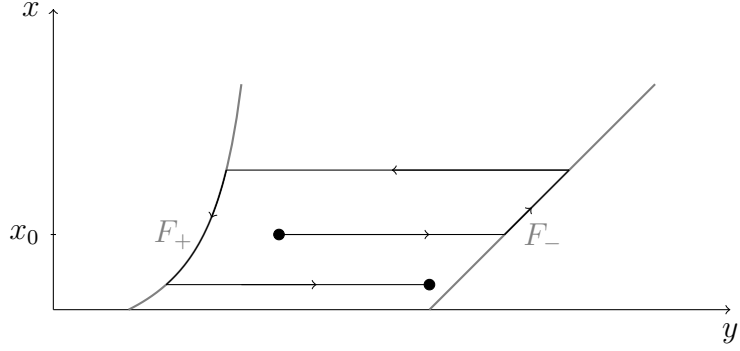
\begin{figure}
		\center
		\begin{tikzpicture}[] 
		\draw[->] (0,0) -- (9,0) coordinate (x axis);
		\draw (-1pt,4 ) node[anchor=east,fill=white] {$x$};
		\draw (-1pt,1 ) node[anchor=east,fill=white] {$x_0$};
		\draw (-1pt,1 )--(1pt,1);
		\draw (9 cm,-1pt) node[anchor=north] {$y$};
		\draw[->] (0,0) -- (0,4) coordinate (x axis);
		\draw [gray, thick, domain=1:2.5] plot (\x, {(\x-1)/(3-\x)});
		\draw [gray] (2 cm,1cm) node[anchor=east] {$F_+$};
		\draw [<-] [black, domain=2.1:2.3] plot (\x, {(\x-1)/(3-\x)});
		\draw [-] [black, domain=1.5:2.3] plot (\x, {(\x-1)/(3-\x)});
		\draw (1.5,1/3) -- (5,1/3);
		\draw [gray, thick, domain=5:((2.5-1)/(3-2.5)+5)] plot (\x, {(\x-5});
		\draw [gray] (6cm +3pt,1cm-1pt) node[anchor=west] {$F_-$};
		\draw [-] [black, domain=6:5+1.3/0.7] plot (\x, {(\x-5});
		\draw [->] [black, domain=6:4.5+1.3/0.7] plot (\x, {(\x-5});
		\draw [<-] [black] (4.5,1.3/0.7) -- (5+1.3/0.7,1.3/0.7);
		\draw [-] [black] (2.3,1.3/0.7) -- (5+1.3/0.7,1.3/0.7);
		\filldraw (5,1/3) circle (2pt);
		\draw [-] [black] (3,1) -- (6,1);
		\filldraw (3,1) circle (2pt);
		\draw [->] [black] (3,1) -- (5,1);
		\draw [->] [black] (2.5,1/3) -- (3.5,1/3);
		\end{tikzpicture}
		\caption{\label{fig:03}Typical finite-time trajectory of $y$ and $x(y,x_0)$ in the 
		$(y,x)$-plane. The initial and final points are marked with black dots. Note that
		we allow quite general curves $F_+,F_-$ as long as the graphs are increasing and 
		Lipschitz.}
	\end{figure}

The class of hysteresis operators represents merely a part of the more general class 
of rate-independent processes. There are various (often equivalent) ways to represent 
such processes. In this paper, we will be concerned with so-called scalar generalized 
play operators which we introduce in the following~\cite[Chapter~III.2]{Visintin}. Let 
$F_-,F_+:\mathbb{R}\rightarrow\mathbb{R}$ be two increasing functions with $F_-\leq F_+$.
If $F_-$ and $F_+$ are Lipschitz continuous, then one way to represent the solution 
$x=x(y,x_0)\in \mathrm{W}^{1,p}(0,T)$ of the generalized play operator, which corresponds 
to the functions $F_-,F_+$ with input $y\in \mathrm{W}^{1,p}(0,T)$, $1\leq p \leq\infty$, 
and $x_0\in \mathbb{R}$, is the solution of the following variational inequality:
\begin{equation}
\begin{aligned}
&\dot{x}(t)(x(t)-\xi) \leq 0 && \forall\xi\in [F_-(y(t)), F_+(y(t))],\ 
\text{a.e. in } (0,T),\\
&x(t)\in [F_-(y(t)), F_+(y(t))] && \text{in }[0,T],\\
&x(0) = \min \{\max \{F_-(y(0)) , x_0\} ,F_+(y(0))\}.
\end{aligned}\label{eq:scalar_play_variational_inequality}
\end{equation}
We do not list all the properties of a generalized play operator here, but refer to the 
literature for certain features once we need them. The typical behaviour of the solution 
$x(y,x_0)$ is depicted in Figure~\ref{fig:03}. The appearance of rate-independent 
systems as the (singular) limit of regularized problems has been analyzed via 
several approaches. We only refer to several results, which are related to our work, 
but emphasize that many other research directions in
the area of hysteresis operators are currently being pursued.\medskip

The first equation in~\eqref{eq:fs} given by 
\be
\label{eq:justfast}
\varepsilon \dot{x} = f(x,y;\varepsilon)
\ee
can be viewed as one possible 
regularization of~\eqref{eq:scalar_play_variational_inequality}.
Equivalent to~\eqref{eq:scalar_play_variational_inequality}, $x(y,x_0)$ 
can be represented by a differential inclusion, see e.g.~\cite{Visintin} 
or~\cite{BrokateSprekels}. Many rate-independent processes are equally 
described by an energetic formulation via a (local/global) stability condition 
and an energy balance condition \cite[Chapter 2 or Chapter 3]{MielkeRoubicek}. 
This leads to the notion of local/global energetic solutions.
There are several other concepts used to describe rate-independent 
processes. We refer to~\cite{Mielke2} or~\cite{MielkeRossiSavare} for overviews.\medskip

In the context of energetic solutions, and the representation of the 
corresponding solutions by rate-independent differential inclusions, 
one often considers the so-called vanishing-viscosity limit, see~\cite{MielkeRossiSavare}. 
The latter is achieved by regularizing the rate-independent differential 
inclusion, frequently by a term modelling viscosity. The solutions of 
$\varepsilon$-dependent regularized problems are then proved to converge 
to an energetic solution of the initial problem. 
One goal of the approach is to reveal properties of the energetic solution such as the 
behaviour at discontinuity points, see~\cite{MielkeRossiSavare}.\medskip

Consider~\eqref{eq:justfast} with a given function $y$. Special choices of $f$, 
$F_+$ and $F_-$ lead to a convex and/or coercive energy functional for the 
regularized, as well as for the limit problem. In this case, 
uniform-in-$\varepsilon$-estimates of the norm of $x_\varepsilon$ by the norm 
of $y$ or $\dot{y}$ in the appropriate spaces can be 
derived \cite[Chapter 1.7 or Chapter 3.8]{MielkeRoubicek} or~\cite{MielkeRossiSavare}.
Moreover, either (I) an equi-continuity estimate for $\{x_\varepsilon\}_\varepsilon$, 
(II) a uniform-in-$\varepsilon$ bound of (some) norm of $\dot{x}_\varepsilon$, or 
(III) a bound of the total variation of $\{x_\varepsilon\}_\varepsilon$ by $y$ and $\dot{y}$ 
can be derived. For (I), a suitable application of the Arzel\`{a}-Ascoli 
Theorem yields the desired convergence of $x_\varepsilon$ to a singular limit $x(y,x_0)$.
For (II) or (III), a weak compactness argument together with Helly's selection theorem 
can be used to prove convergence. For \emph{special choices} of $f$, $F_+$ and $F_-$, 
it should be possible to use them also for the coupled system \eqref{eq:fs}. Our 
situation may include non-convex and non-coercive energy functionals as well as 
systems \emph{without any energy structure}.\medskip

Furthermore, in our setting, in the formulation as variational inequalities, the generality 
of $f$, $F_+$ and $F_-$ together with the coupling in the fast-slow system~\eqref{eq:fs} 
complicates the limit procedure, since uniform-in-$\varepsilon$-estimates for 
$\dot{x}_\varepsilon$, even in $L^1(0,T)$, can no longer be derived. We bypass this 
problem by \emph{projecting} $x_\varepsilon$ in $y$-$x$-phase space vertically onto
the set $\mathcal{C}_0$ of roots of $f$, i.e., we employ the geometric one-sided limit
available for the critical manifold as discussed above. For the projection $p_\varepsilon$ 
we can even show $\mathrm{W}^{1,\infty}(0,T)$-bounds, which are uniform in $\varepsilon$. 
That is, the question of convergence of $x_\varepsilon$ and $y_\varepsilon$ to $x(y,x_0)$ 
and $y$ is shifted to the problem of showing that $x_\varepsilon - p_\varepsilon$ converges 
to zero in the limit $\varepsilon\rightarrow 0$ in an appropriate space, 
and that the limit $x(y,x_0)=\lim_{\varepsilon\ra 0}p_\varepsilon$ follows the 
hysteresis law \eqref{eq:scalar_play_variational_inequality}, where $y=\lim_{\varepsilon\ra 0}y_\varepsilon$ 
solves the fast equation in \eqref{eq:fs} with $x=x(y,x_0)$. The latter can be proved in our 
setting.\medskip 

Another result about hysteresis as the singular limit in ODEs, which can be compared to 
our problem is derived in~\cite{Krejci2} assuming the 
\emph{Li\'enard case} 
\benn
f(x,y;\varepsilon)= -F(x)+y.
\eenn
The source function $y$ is a priori \emph{fixed}, i.e., independent of $\varepsilon$ 
and independent of fast-slow coupling. Furthermore, the set $\mathcal{C}_0$ of this 
problem has \emph{co-dimension $1$}, different from our setting, where major difficulties 
appear as $\mathcal{C}_0$ has co-dimension $0$. Another main difference in~\cite{Krejci2} 
is that the singular limit $x$ of $x_\varepsilon$ with $y$ independent of $\varepsilon$ is the 
output of a hysteresis operator of \emph{switch type} with input $y$. Since the solutions 
$x_\varepsilon$ are continuous, while the limit $x$ is in general discontinuous, uniform 
convergence in classical spaces cannot be expected 
in~\cite{Krejci2}. However, uniform
convergence results have been obtained for the concept of $r$-convergence in the space of 
regulated functions involving uniform bounds. We completely bypass the problem of showing 
uniform-in-$\varepsilon$ bounds for the oscillations or for the derivative of the 
solutions $x_\varepsilon$ in \eqref{eq:fs} by introducing suitable \emph{projection functions} 
$p_\varepsilon$ below.\medskip

In fact, our approach substantially differs from previous approaches already by its fundamental
principles and the direct relation to the fast-slow GSPT viewpoint. It also provides several
new technical tools, and it does not require specialized spaces, ODE
structure assumptions, or the existence of an energy. In particular, this setting is designed to
link GSPT and fast-slow systems a lot more directly to hysteresis limits than has been 
possible so far.
	
\section{The Main Result I - Functional Approach}
\label{sec:main1}

The main system we are going to study is a planar fast-slow system with the added 
generalization that the slow vector field may also depend upon time
\begin{alignat}{2}
	\varepsilon \dot{x} &= f(x,y),\label{Eq:evol_epsilon_y}\\
	\dot{y} &= g(x,y,t),\label{Eq:evol_epsilon_x}
\end{alignat}
augmented with the initial condition $(x(0),y(0)) = (x_0,y_0)$. Our
argument also allows for the generalization $g(x,y,t)=g(x,y,t;\varepsilon)$ and 
$f(x,y)=f(x,y;\varepsilon)$ as long as $f,g$ are bounded for all $\varepsilon
\in[0,\varepsilon_0]$ for some sufficiently small $\varepsilon_0>0$. We shall 
not make this additional generalization explicit and just omit the 
$\varepsilon$-dependence in $f,g$. We denote the solutions 
of~\eqref{Eq:evol_epsilon_y}-\eqref{Eq:evol_epsilon_x} by 
$(x_\varepsilon,y_\varepsilon)=(x_\varepsilon(t),y_\varepsilon(t))$ to emphasize 
the dependence upon $\varepsilon$. We are interested in the limit as 
$\varepsilon \ra 0$ on the finite time interval $J_T:=(0,T)$. The situation outlined 
already in Sections~\ref{ssec:fastslow}-\ref{ssec:hysteresis} is made precise by the 
following main assumptions on $f,g$:

\begin{itemize}
		\item[\mylabel{a:fg}{(A1)}] $f,g$ are Lipschitz continuous with 
		Lipschitz constants $L_f$ and $L_g$ respectively; we can also allow
		only local Lipschitz continuity in our results but refrain from 
		doing so as it just complicates the notation.
		\item[\mylabel{a:cm}{(A2)}] The function $f$ satisfies
			\begin{align}
					\begin{matrix}
						f(x,y) < 0 &&\text{ if } x > F_+(y),\\
						f(x,y) = 0 &&\text{ if } x \in [F_-(y),F_+(y)],\\
						f(x,y) > 0 && \text{else}.						
					\end{matrix}\label{Eq:Ass_g}
			\end{align}
			for two functions $F_+ > F_-$.
		\item[\mylabel{a:Fpm}{(A3)}]	$F_+,F_-$ are Lipschitz continuous with 
		Lipschitz constants $L_+$ and $L_-$; we define $L_\pm:=\max\{L_+,L_-\}$. 
		$F_+,F_-$ are monotone increasing functions.
		\item[\mylabel{a:gbound}{(A4)}] $g$ satisfies the growth assumption
			\begin{align}
				|g(x,y,t) | \leq M(t) ( 1+ \cG(x) + |y| )\label{Eq:growth_f}
			\end{align}
			with $M>0$ continuous and $\cG$ bounded.
\end{itemize}

Assumption~\ref{a:fg} essentially assures the existence of unique (local) solutions 
of the system \eqref{Eq:evol_epsilon_y}-\eqref{Eq:evol_epsilon_x}. \ref{a:cm} forces 
the solutions $(x_\varepsilon,y_\varepsilon)$ to converge to the critical manifold 
$\cC_0$ as $\varepsilon\rightarrow 0$, in particular, the signs of $f$ are given so
that $\cC_0$ can be viewed as attracting with respect to any fast trajectory movement. 
Assumption~\ref{a:Fpm} is going to yield the required Lipschitz continuity of the 
corresponding limiting generalized play operator and enables us to carry over Lipschitz 
bounds of $y_\varepsilon$, which are independent of $\varepsilon$, to $F_+(y_\varepsilon)$ 
and $F_-(y_\varepsilon)$. Amongst others, \ref{a:cm} and \ref{a:Fpm} together are 
eventually going to allow us to show bounds of $x_\varepsilon$ in 
$\mathrm{C}(\overline{J_T})=\mathrm{C}^0(\overline{J_T})$, which are independent 
of $\varepsilon$. Assumption~\ref{a:gbound} yields that $y_\varepsilon$ is bounded in 
$\mathrm{W}^{1,\infty}(J_T)$ independently of $\varepsilon$. In particular, it provides
a growth bound of the slow variable, which could just become unbounded, i.e., other 
growth bounds are expected to work as well.\medskip

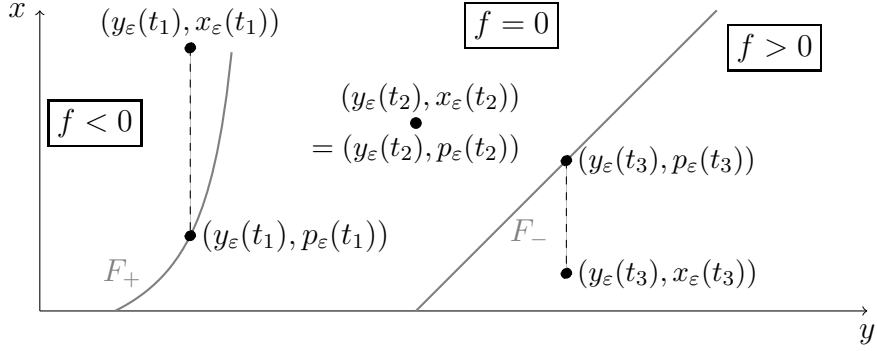
\begin{figure}
\center
\begin{tikzpicture}[]
\draw[->] (0,0) -- (11,0) coordinate (x axis);
\draw (-1pt,4 ) node[anchor=east,fill=white] {$x$};
\draw (11 cm,-1pt) node[anchor=north] {$y$};
\draw[->] (0,0) -- (0,4) coordinate (x axis);
\draw [gray, thick, domain=1:2.55] plot (\x, {(\x-1)/(3-\x)});
\draw [gray] (1.5 cm ,0.5cm) node[anchor=east] {$F_+$};
\draw [gray, thick, domain=5:9] plot (\x, {(\x-5});
\draw [gray] (6cm +3pt,1cm+1pt) node[anchor=west] {$F_-$};
\draw  [black] (5.5,3.8) node[anchor=west] {\fbox{$f=0$}};
\draw  [black] (1.5,2.5) node[anchor=east] {\fbox{$f<0$}};
\draw  [black] (9,3.5) node[anchor=west] {\fbox{$f>0$}};
\filldraw [dashed]
(2,3.5) circle (2pt) node[align=left, above] 
{\small $(y_\varepsilon(t_1),x_\varepsilon(t_1))$}--
(2,1) circle (2pt) node[align=left, right] {$(y_\varepsilon(t_1),p_{\varepsilon}(t_1))$};
\filldraw [dashed]
(7,2) circle (2pt) node[align=left, right] 
{\small $(y_\varepsilon(t_3),p_{\varepsilon}(t_3))$}--
(7,0.5) circle (2pt) node[align=left, right] 
{\small $(y_\varepsilon(t_3),x_{\varepsilon}(t_3))$};
\filldraw (5,2.5) circle (2pt) node[align=left, above] 
{\small $\ \ \ (y_\varepsilon(t_2),x_\varepsilon(t_2))$};
\draw (5,2.5) node[align=left, below] {\small $=(y_\varepsilon(t_2),p_\varepsilon(t_2))$};
\end{tikzpicture}
\caption{Sign behaviour of $f$ and projection to $p_\varepsilon$ in the $(y,x)$-plane.
Note that the critical manifold $\cC_0=\{f=0\}$ is two-dimensional and bounded by the 
two curves $F_\pm$; see also assumption~\ref{a:cm}. The projection along the vertical
fast coordinate to $\cC_0$ is denoted by $p_\varepsilon$ and only acts on the fast 
variable.}
\label{fig:04}
\end{figure}

Since the derivative of $x_\varepsilon$ cannot be bounded independently of $\varepsilon$, 
we introduce a projection function for which such a bound can be derived. As shown in 
Figure~\ref{fig:04}, let $p_{\varepsilon}=p(x_\varepsilon,y_\varepsilon)$ be 
defined by
\begin{alignat*}{2}
p(x_\varepsilon,y_\varepsilon):= \min\{\max\{
x_\varepsilon, F_-(y_\varepsilon)\},F_+(y_\varepsilon)\}.
\end{alignat*}
The family of projections $p_\varepsilon$ is continuous because $F_-$,$F_+$, $x_\varepsilon$, 
$y_\varepsilon$, $\max$ and $\min$ are all continuous. The main theorem we are going to 
prove in several steps below is the following:

\begin{thm}
\label{thm:projection}
Suppose the assumptions~\ref{a:fg}-\ref{a:gbound} hold and let $q\in(1,+\I)$. Denote 
by $(\overline{x},\overline{y})$ the unique solution of
\begin{alignat}{2}
&\dot{y}(t) = g(x(t),y(t),t)\ && \text{a.e. in } J_T,\label{Eq:evol_x_1}\\
& y(0) = y_0,\label{Eq:evol_x_2}\\
&\dot{x}(t)(x(t)-\xi) \leq 0 && \forall\xi\in [F_-(y(t)), F_+(y(t))],\ 
\text{a.e. in } J_T,\label{Eq:evol_y_1}\\
&x(0) = \min \{\max \{F_-(y_0) , x_0\} ,F_+	(y_0)\},\label{Eq:evol_y_2}\\
&x(t)\in [F_-(y(t)), F_+(y(t))] && 
\text{in }\overline{J_T}.\label{Eq:evol_y_3}
\end{alignat}
Then we have $\overline{x}\in \mathrm{W}^{1,\infty}(J_T)$ and $\overline{y}\in 
\mathrm{C}^1(J_T)\cap \mathrm{C}(\overline{J_T})$. Moreover, the main convergence result is
\benn
x_\varepsilon \rightarrow \overline{x} \text{ in } \mathrm{L}^{q}(J_T)	\text{ and }
y_{\varepsilon} \rightarrow \overline{y} \text{ in } \mathrm{W}^{1,q}(J_T)
\eenn
as $\varepsilon\rightarrow 0$, i.e., the singular limit $\overline{x}$ is the solution 
of a generalized play operator for the curves $F_+$ and $F_-$ with input $\overline{y}$. 
\end{thm}

Theorem~\ref{thm:projection} shows that the singular limit of the non-autonomous planar 
fast-slow system is~\eqref{Eq:evol_epsilon_y} with~\eqref{Eq:evol_epsilon_x} replaced by 
a hysteresis operator. An example of a trajectory of a generalized scalar play operator is shown in 
Figure~\ref{fig:03}. Note that the time-dependence of the slow vector field can indeed
generate highly nontrivial dynamics inside $\cC_0$. Of course, Theorem~\ref{thm:projection}
only provides a partial solution of Netushil's conjecture as we have not characterized 
\emph{all classes} of hysteresis operators, which arise as singular fast-slow limits. 
We need to derive several important auxiliary results before we can prove 
Theorem~\ref{thm:projection}.

\begin{lem}
\label{lem:boundedness}
Fix $\varepsilon>0$ and assume~\ref{a:fg}-\ref{a:gbound} hold. Then we can bound 
$y_{\varepsilon}$ in $\mathrm{C}(\overline{J_T})\cap\mathrm{C}^1(J_T)$ and $p_{\varepsilon}$ 
in $\mathrm{W}^{1,\infty}(J_T)$ independent of $\varepsilon$. More precisely, there exists
a constant $K_1>0$ such that 
\be
\label{eq:fboundlem1}
\|y_\varepsilon\|_{\mathrm{C}(\overline{J_T})}  + \|y_\varepsilon\|_{\mathrm{C}^1(J_T)} 
+ \|p_\varepsilon\|_{\mathrm{W}^{1,\infty}(J_T)} \leq K_1.
\ee
More precisely, there exists a constant $K_2>0$ such that
\be
	\|x_\varepsilon - p_\varepsilon\|_{\mathrm{C}(\overline{J_T})} 
	\leq K_2.\label{Eq:estim_dist_y_eps_p_eps}
\ee
The constants $K_1, K_2>0$ are	independent of $\varepsilon$.	
\end{lem}

\begin{proof}
For this proof, let $c>0$ denote a generic constant, which does not depend on 
$\varepsilon$. First, we are going to show that $y_{\varepsilon}$ is bounded 
in $\mathrm{C}(\overline{J_T})\cap\mathrm{C}^1(J_T)$. In a second step, we prove that 
$p_{\varepsilon}$ is bounded in $\mathrm{W}^{1,\infty}(J_T)$. By assumption~\ref{a:gbound}
we obtain an estimate of the form
\be
\label{eq:estslow1}
|y_\varepsilon(t)|\leq |y_0| + \int_0^t M(s) (1 + \cG(x_\varepsilon(s)) 
+ |y_\varepsilon(s)|) ~\txtd s \leq c_0 + c_1 \int_0^t |y_\varepsilon(s)|~\txtd s
\ee
for some constants $c_0,c_1>0$; note that the constants in \eqref{eq:estslow1} may grow
in time as $M$ is time-dependent but the constants remain finite for every fixed final 
time $T>0$. Applying Gronwall's Lemma to~\eqref{eq:estslow1} implies 
$\|y_\varepsilon\|_{\mathrm{C}(\overline{J_T})} \leq c$. Equation~\eqref{Eq:evol_epsilon_x} 
and assumption~\ref{a:gbound} yield $\|y_\varepsilon\|_{\mathrm{C}^1(J_T)} \leq c$. For 
further use we note that this implies that $y_\varepsilon$ is Lipschitz continuous on 
$J_T$ with a Lipschitz constant independent of $\varepsilon$. Regarding $p_\varepsilon$, 
it follows from the definition of the projection that
\be
\label{eq:defnpro}
p_\varepsilon(t) \in [F_-(y_\varepsilon(t)) , F_+(y_\varepsilon(t))].
\ee
Since $F_-$ and $F_+$ are Lipschitz continuous by~\ref{a:Fpm}, we obtain 
from~\eqref{eq:defnpro} that $\|p_\varepsilon\|_{\mathrm{C}(\overline{J_T})} \leq c$. It 
remains to to show that the norm of $p_\varepsilon$ in $\mathrm{W}^{1,\infty}(J_T)$ can 
be bounded independently of $\varepsilon$. Again from the definition of $p_\varepsilon$ 
it follows for a.e.~$t\in J_T$ that
\benn
p_\varepsilon(t) \in (F_-(y_\varepsilon(t)) , F_+(y_\varepsilon(t))) \quad 
\text{if and only if}\quad  \dot{p}_\varepsilon(t)=0=\dot{x}_\varepsilon(t),
\eenn
and otherwise, $p_\varepsilon(t) \in \{F_-(y_\varepsilon(t)),F_+(y_\varepsilon(t))\}$.
Assumption~\ref{a:Fpm} together with the bounds for $y_\varepsilon$, which we have shown 
already, yield that $F_-(y_\varepsilon(\cdot))$ and $F_+(y_\varepsilon(\cdot))$ are Lipschitz 
continuous with Lipschitz constant independent of $\varepsilon$. Let $t_1<t_2$ be given and 
suppose $p_\varepsilon(t_1)=F_+(y_\varepsilon(t_1))$. If $p_\varepsilon(t)=
F_+(y_\varepsilon(t))$ in the time interval $[t_1,t_2]$ then
\benn
|p_\varepsilon(t_1) - p_\varepsilon(t_2)|=  
|F_+(y_\varepsilon(t_1))-F_+(y_\varepsilon(t_2))| \leq c |t_1-t_2|
\eenn
by Lipschitz continuity of $F_+(y_\varepsilon(\cdot))$. Otherwise, there exist times 
$t^{(1)}\in [t_1,t_2)$ and $t^{(2)}\in (t^{(1)},t_2]$ such that 
\begin{alignat*}{2}
x_\varepsilon(t)&\geq F_+(y_\varepsilon(t))\ &&\forall\, t\in [t_1,t^{(1)}] \text{ and }\\
x_\varepsilon(t)&=F_+(y_\varepsilon(t^{(1)}))\quad &&\forall\, t\in [t^{(1)},t^{(2)}].
\end{alignat*}
Note that in the case when $t_1=t^{(1)}$ we imply the empty set if we write 
$[t_1,t^{(1)}]$. In this setting we find
\benn
|p_\varepsilon(t_1) - p_\varepsilon(t^{(2)})|=  
|F_+(y_\varepsilon(t_1))-F_+(y_\varepsilon(t^{(1)}))| \leq 
c |t_1-t^{(1)}| \leq c |t_1-t^{(2)}|.
\eenn
We first choose $t^{(1)}\in [t_1,t_2)$ and then $t^{(2)}\in (t^{(1)},t_2]$ both maximal.
If $t^{(2)}\neq t_2$ and $F_+(y_\varepsilon(t^{(1)}))=F_-(y_\varepsilon(t^{(2)}))$ then 
we apply the same reasoning with $F_-$ on the interval $[t^{(2)},t_2]$. We obtain a 
partition $t^{(0)}=t_1 \leq t^{(1)} \leq \cdots \leq t^{(k)}=t_2$ of $[t_1,t_2]$ such 
that
\benn
|p_\varepsilon(t_1) - p_\varepsilon(t_2)|
\leq \sum_{i=1}^{k}  |p_\varepsilon(t^{(i-1)}) - p_\varepsilon(t^{(i)})|
\leq c\sum_{i=1}^{k}  |t^{(i-1)} - t^{(i)}| = c |t_1-t_2|.
\eenn
This implies that $p_\varepsilon$ is Lipschitz continuous independent of $\varepsilon$. 
By a standard embedding theorem~\cite[Chapter~5.8.2.b,~Theorem 4]{Evans} it follows
that
\benn
\|p_\varepsilon\|_{\mathrm{W}^{1,\infty}(J_T)} \leq c.
\eenn
This concludes the proof of the first bound~\eqref{eq:fboundlem1} in the result. 
It remains to show that $\|x_\varepsilon - p_\varepsilon\|_{\mathrm{C}(\overline{J_T})}$ 
is bounded. We calculate
\begin{align*}
&|x_\varepsilon(t)-p_{\varepsilon}(t)|\\
&=|x_\varepsilon(0)-p_{\varepsilon}(0)| + \int_0^t (\dot{x}_\varepsilon(s) 
- \dot{p}_{\varepsilon}(s)) \frac{x_\varepsilon(s)-p_{\varepsilon}(s)}{|x_\varepsilon(s)
- p_{\varepsilon}(s)|} \, \txtd s\\
&= |x(0)-p(0)| 
+ \frac{1}{\varepsilon}\int_0^t f( x_\varepsilon(s) , y_\varepsilon(s) ) 
\frac{x_\varepsilon(s)-p_{\varepsilon}(s)}{|x_\varepsilon(s)-p_{\varepsilon}(s)|} \,\txtd s 
- \int_0^t \dot{p}_{\varepsilon}(s) 
\frac{x_\varepsilon(s)-p_{\varepsilon}(s)}{|x_\varepsilon(s)-p_{\varepsilon}(s)|} \,\txtd s.
\end{align*}
By the stability assumption~\ref{a:cm} and the definition of $p_\varepsilon$ it holds
\be
\label{Eq:leq_zero_estimate}
f( x_\varepsilon(s) , y_\varepsilon(s) ) 
\frac{x_\varepsilon(s)-p_{\varepsilon}(s)}{|x_\varepsilon(s)-p_{\varepsilon}(s)|}\leq 0
\ee
for all $s\in [0,t]$, cf.~Figure~\ref{fig:04}. Hence, it follows that 
\be
\label{eq:pop_dnymics_stock_estimate}
|x_\varepsilon(t)-p_\varepsilon(t)| - \frac{1}{\varepsilon}\int_0^t 
f( x_\varepsilon(s) , y_\varepsilon(s) ) 
\frac{x_\varepsilon(s)-p_{\varepsilon}(s)}{|x_\varepsilon(s)-p_{\varepsilon}(s)|}\,\txtd s 
\leq |x(0)-p(0)| + \int_0^t |\dot{p}_{\varepsilon}(s)|\, \txtd s
\ee
and the left side is greater or equal than zero. We have already shown that the right 
side in~\eqref{eq:pop_dnymics_stock_estimate} is bounded independently of $\varepsilon$. 
Consequently, the estimate~\eqref{eq:pop_dnymics_stock_estimate} implies
\benn 
\|x_\varepsilon - p_\varepsilon\|_{\mathrm{C}(\overline{J_T})} \leq c, 
\eenn
which finishes the proof.
\end{proof}

Since the bounds in Lemma~\ref{lem:boundedness} are independent of $\varepsilon$, we 
can proceed by a compactness argument and prove convergence results for appropriate 
subsequences.

\begin{lem}
\label{lem:convergence}
Fix any $q\in(1,+\infty)$ and assume~\ref{a:fg}-\ref{a:gbound}, then the following 
results hold:
\begin{itemize}
 \item[\mylabel{c:xp}{(C1)}] $\lim_{\varepsilon\ra 0}\|x_\varepsilon - 
p_\varepsilon\|_{\mathrm{L}^q(J_T)}=0$.
 \item[\mylabel{c:weakstrong}{(C2)}] Given any sequence $\{\varepsilon_j\}_{j=1}^\I$ 
such that $\varepsilon_j\ra 0$ as $j\ra +\I$, there exists a subsequence 
$\{\varepsilon_{j_k}=:\varepsilon_k\}_{k=1}^\I$ and functions 
$\overline{x},\overline{y}\in \mathrm{W}^{1,q}(J_T)$ such that
	\be
	\label{Conv_subseq_p_and_x}
	p_{\varepsilon_k} \ra \overline{x} \quad  \text{ and } \quad 
	y_{\varepsilon_k} \ra \overline{y}
	\ee
	as $\varepsilon_k \rightarrow 0$ weakly in $\mathrm{W}^{1,q}(J_T)$ and strongly 
	in $\mathrm{C}(\overline{J_T})$.
\item[\mylabel{c:yybar}{(C3)}]	In $\mathrm{L}^q(J_T)$ as $\varepsilon_k \rightarrow 0$, 
we have
	\be
	\label{Conv_subseq_y}
	x_{\varepsilon_k} \rightarrow \overline{x}.
	\ee
	\end{itemize}
\end{lem}

\begin{proof}
In order to prove~\ref{c:xp}, we compute for arbitrary $\varepsilon>0$ and 
$t\in\overline{J_T}$
\begin{align*}
(x_\varepsilon(t)-p_{\varepsilon}(t))^2 &= (x(0)-p(0))^2 
+ 2\int_0^t (\dot{x}_\varepsilon(s) - \dot{p_{\varepsilon}}(s)) (x_\varepsilon(s)
-p_{\varepsilon}(s)) \, \txtd s\\
&= (x(0)-p(0))^2 + \frac{2}{\varepsilon}\int_0^t f(x_\varepsilon(s),y_\varepsilon(s)) 
(x_\varepsilon(s)-p_{\varepsilon}(s))\, \txtd s\\
&\quad -2\int_0^t \dot{p}_{\varepsilon}(s)~ (x_\varepsilon(s)-p_{\varepsilon}(s))~\txtd s.
\end{align*}
This calculation together with Lemma~\ref{lem:boundedness} 
and~\eqref{Eq:leq_zero_estimate} yields
\be
\label{eq:boundquad}
0\leq (x_\varepsilon(t)-p_{\varepsilon}(t))^2 - 
\frac{2}{\varepsilon}\int_0^t f( x_\varepsilon(s) , y_\varepsilon(s) ) 
(x_\varepsilon(s)-p_{\varepsilon}(s))\, \txtd s\leq c,
\ee
where $c>0$ is, as before, a generic constant independent of $\varepsilon$.	The bounds 
in~\eqref{eq:boundquad} and the sign condition~\eqref{Eq:leq_zero_estimate} imply
\benn
f( x_\varepsilon(s) , y_\varepsilon(s) )(x_\varepsilon(s)-p_{\varepsilon}(s))\ra 0
\eenn
for a.e.~$s\in J_T$ as $\varepsilon\rightarrow 0$. By definition of $p_\varepsilon$ and 
assumption~\ref{a:cm} we conclude that $x_\varepsilon(s)-p_{\varepsilon}(s)$ tends to 
zero for a.e.~$s\in J_T$ as $\varepsilon\rightarrow 0$. This result, together 
with~\eqref{Eq:estim_dist_y_eps_p_eps} and the Lebesgue dominated convergence theorem, 
implies
\benn
\lim_{\varepsilon\ra 0}\|x_\varepsilon - p_\varepsilon \|_{\mathrm{L}^q(J_T)} = 0,
\eenn
for any $q\in(1,+\I)$, which concludes the proof of~\ref{c:xp}. It remains to show  
\ref{c:weakstrong}-\ref{c:yybar} concerning subsequences for a given sequence 
$\{\varepsilon_j\}_{j=1}^\I$ with $\lim_{j\ra \I}\varepsilon_j=0$. By 
Lemma~\ref{lem:boundedness} the functions $y_{\varepsilon}$ and $p_{\varepsilon}$ are 
bounded in $\mathrm{W}^{1,\infty}(J_T)$ independently of $\varepsilon$ and hence are
in $\mathrm{W}^{1,q}(J_T)$ for any $q\in(1,+\I]$. Using the Banach-Alaoglu
Theorem~\cite{RudinFunc}, it follows that there is a subsequence 
$\{\varepsilon_{j_k}=:\varepsilon_k\}_{k=1}^\I$ of $\{\varepsilon_j\}_{j=1}^\I$ and 
that there exist functions $\overline{x},\overline{y}\in \mathrm{W}^{1,q}(J_T)$ such 
that
\be
\label{Conv_subseq_p_and_x1}
p_{\varepsilon_k} \rightharpoonup \overline{x} \quad 
\text{ and } \quad y_{\varepsilon_k} \rightharpoonup \overline{y}
\ee
as $\varepsilon_k \rightarrow 0$ weakly in $\mathrm{W}^{1,q}(J_T)$. Because 
$\mathrm{W}^{1,q}(J_T)$ is compactly embedded in $\mathrm{C}(\overline{J_T})$ for 
$1<q\leq \infty$ \cite[Theorem~6.3]{AdamsFournier}, this convergence is strong in 
$\mathrm{C}(\overline{J_T})$. Since $x_\varepsilon - p_\varepsilon \rightarrow 0$ 
in $\mathrm{L}^q(J_T)$ we finally conclude $x_{\varepsilon_k} \ra \overline{x}$
in $\mathrm{L}^q(J_T)$ as $\varepsilon_k \rightarrow 0$.
\end{proof}

Having proven that subsequences of $x_\varepsilon$, $y_\varepsilon$ and $p_\varepsilon$ 
actually converge in a certain sense, we would like to understand the behaviour of the 
limit functions. We can also improve the type of convergence of the subsequence 
$\{y_{\varepsilon_k}\}_{k=1}^\I$. 

\begin{lem}
\label{lem:Limit_equations}
Consider the same assumptions and the notation of Lemma~\ref{lem:convergence}.
The functions $\overline{x}$ and $\overline{y}$ solve the system 
\begin{alignat}{2}
	&\dot{y}(t) = g(x(t),y(t),t)\ && \text{a.e. in } J_T,\label{Eq:evol_x_1_Lem}\\
	& y(0) = y_0,\label{Eq:evol_x_2_Lem}\\
	&\dot{x}(t)(x(t)-\xi) \leq 0 && \forall\xi\in [F_-(y(t)), F_+(y(t))],\ 
	\text{a.e. in } J_T,\label{Eq:evol_y_1_Lem}\\
	&x(0) = \min \{\max \{F_-(y_0) , x_0\} ,F_+	(y_0)\},\label{Eq:evol_y_2_Lem}\\
	&x(t)\in [F_-(y(t)), F_+(y(t))] && 
	\text{in }\overline{J_T}.\label{Eq:evol_y_3_Lem}
\end{alignat}
Furthermore, $y_{\varepsilon_k}\rightarrow \overline{y}$ in $\mathrm{W}^{1,q}(0,T)$ 
as $k\rightarrow \infty$, i.e., we have strong convergence in $\mathrm{W}^{1,q}(0,T)$
of the slow dynamics for a subsequence.
\end{lem}

\begin{proof}
We first show that $\overline{y}$ solves~\eqref{Eq:evol_x_1_Lem}-\eqref{Eq:evol_x_2_Lem} 
with $x=\overline{x}$ and improve the convergence of $\{y_{\varepsilon_k}\}_{k=1}^\I$.
Lemma~\ref{lem:convergence}, assumptions \ref{a:fg} and \ref{a:gbound}, and the Lebesgue 
dominated convergence theorem yield that
\be
\label{Eq:strong_conv_deriv_x}
g(x_{\varepsilon_k},y_{\varepsilon_k},t) \ra g(\overline{x},\overline{y},t)
\ee
in $\mathrm{L}^q(J_T)$ as $\varepsilon_k \rightarrow 0$. By Lemma~\ref{lem:convergence}, 
$y_{\varepsilon_k}$ converges to $\overline{y}$ uniformly in $\overline{J_T}$. For 
$t\in \overline{J_T}$ we obtain by using $y_{\varepsilon_k} = y_0 + \int_{0}^{t} 
g(x_{\varepsilon_k},y_{\varepsilon_k},s) \, \txtd s$ that
\be
\left|\overline{y}(t) - y_0 - \int_{0}^{t} g(\overline{x},\overline{y},s) \, \txtd s\right|
\leq | \overline{y}(t) - y_{\varepsilon_k}(t) |
+ \int_{0}^{t} | g(\overline{x},\overline{y},s) 
- g(x_{\varepsilon_k},y_{\varepsilon_k},s)| \, \txtd s
\rightarrow 0 
\ee
as $\varepsilon_k \rightarrow 0$ by using Cauchy-Schwarz and $t<+\I$ for the last term
to get convergence in $L^1$. This shows that $\overline{y}$ 
solves~\eqref{Eq:evol_x_1_Lem}-\eqref{Eq:evol_x_2_Lem} with $x=\overline{x}$. Together 
with~\eqref{Eq:strong_conv_deriv_x} we may conclude 
that $y_{\varepsilon_k}\rightarrow \overline{y}$ in $\mathrm{W}^{1,q}(0,T)$.

Next, we are going to show that $\overline{x}$ 
solves~\eqref{Eq:evol_y_1_Lem}-\eqref{Eq:evol_y_3_Lem} with $y=\overline{y}$. First, we 
will deal with~\eqref{Eq:evol_y_2_Lem}-\eqref{Eq:evol_y_3_Lem}. By definition 
of $p_{\varepsilon}$ and with Lemma~\ref{lem:convergence} we have
\be
\label{eq:convinitial}
\overline{x}(0)=\lim_{k\rightarrow \infty}p_{\varepsilon_k}(0)=
\min\{\max\{x_0, F_-(y_0)\},F_+(y_0)\}
\ee
as well as
\be
\label{eq:convdomain}
\overline{x}(t)\in [F_-(\overline{y}(t)), F_+(\overline{y}(t))]\quad  
\forall\ t\in \overline{J_T},
\ee
which proves~\eqref{Eq:evol_y_2_Lem}-\eqref{Eq:evol_y_3_Lem}. Hence, it remains to show
the variational inequality~\eqref{Eq:evol_y_1_Lem}, which we accomplish in two steps.
First, we deal with initial data in the interior of the critical manifold $\cC_0$
and in a second step we are going to consider dynamics on the boundary. Fix $t_0\in 
\overline{J_T}$ and suppose we start in the interior
\benn
\overline{x}(t_0) \in (F_-(\overline{y}(t_0)),F_+(\overline{y}(t_0))).
\eenn
Continuity of $\overline{x}$, $F_-(\overline{y}(\cdot))$ and $F_+(\overline{y}(\cdot))$ 
implies that there is some interval $J\subset J_T$ with $t_0\in J$ such that 
$\overline{x}(t) \in ( F_-(\overline{y}(t)) , F_+(\overline{y}(t)) )$ for all 
$t\in \overline{J}$. We define the distance to the boundary as
\benn
\delta_J:= \min_{t\in J}\{ \overline{x}(t) - F_-(\overline{y}(t)) , 
F_+(\overline{y}(t)) - \overline{x}(t) \}.
\eenn
By Lemma~\ref{lem:convergence} and assumption~\ref{a:Fpm}, we can find some 
$\varepsilon^{(0)} > 0$ such that for all $\varepsilon_k<\varepsilon^{(0)}$ and 
all $t\in \overline{J_T}$ the following estimate holds
\be
\label{eq:simbelow}
|F_-(y_{\varepsilon_k}(t)) - F_-(\overline{y}(t))| + |F_+(y_{\varepsilon_k}(t)) 
- F_+(\overline{y}(t))| + |p_{\varepsilon_k}(t)-\overline{x}(t)|<\frac{\delta_J}{4}.
\ee
This implies $p_{\varepsilon_k}(t) \in 
( F_-(y_{\varepsilon_k}(t)),F_+(y_{\varepsilon_k}(t)) )$ as well as
\benn
	\min_{t\in J}\{ p_{\varepsilon_k}(t) - F_-(y_{\varepsilon_k}(t)) , 
	F_+(y_{\varepsilon_k}(t)) - p_{\varepsilon_k}(t) \} \geq \frac{\delta_J}{2}.
\eenn
for all $\varepsilon_k<\varepsilon^{(0)}$ and $t\in J$. By definition of the 
projection $p_{\varepsilon_k}$ we immediately find $p_{\varepsilon_k}(t)=
x_{\varepsilon_k}(t)$ and $\dot{p}_{\varepsilon_k}(t)=0$ for a.e.~$t\in J$ 
so that $p_{\varepsilon_k}(t)=p_{\varepsilon_k}(t_0)$ for $t\in J$ and the
variational inequality is just satisfied with zero almost everywhere in $J$.
As the second step, suppose we start on the boundary of the critical manifold
which occurs e.g.~if $\overline{x}(t_0) = F_+(\overline{y}(t_0))$. Then there is 
some interval $J$ with $t_0\in J$ such that $\overline{x}(t) > 
F_-(\overline{y}(t))$ for all $t\in \overline{J}$; note that we slightly
overload the notation here and again use $J$. A similar argument leads
to \eqref{eq:simbelow} and the conclusion that for all for all $\varepsilon_k
<\varepsilon^{(0)}$ and $t\in J$ we now have
\benn
p_{\varepsilon_k}(t) > F_-(y_{\varepsilon_k}(t)) + \frac{\delta_J}{2}.
\eenn
By definition of $p_{\varepsilon_k}$ it follows $x_{\varepsilon_k}(t)\geq 
F_-(y_{\varepsilon_k}(t))$ and $\dot{x}_{\varepsilon_k}(t)\leq 0$  for 
a.e.~$t\in J$. In particular, we have a negative sign for 
$\dot{x}_{\varepsilon_k}(t)$, which we want to transfer to the limit
and show that
\be
\label{eq:coolintgoal}
	\dot{\overline{x}} \leq 0,\qquad \text{a.e. in $J$.}
\ee
To prove~\eqref{eq:coolintgoal}, assume in contradiction that
\be
\label{eq:tobecontra}
\overline{x}(t_2) > \overline{x}(t_1)\qquad 
\text{for some $t_1,t_2\in J$, $t_1<t_2$.}
\ee
Then there exists $\varepsilon^{(1)}>0$ such that 
$p_{\varepsilon_k}(t_2)>p_{\varepsilon_k}(t_1)$ if $\varepsilon_k < \varepsilon^{(1)}$.
Using the Lipschitz continuity of the $p_{\varepsilon_k}$ from 
Lemma~\ref{lem:boundedness}, we can find $t^{(1)}\in (t_1,t_2)$ and $\delta_1>0$ 
such that $p_{\varepsilon_k}(t) < p_{\varepsilon_k}(t_2)-\delta_1$ for all 
$t\in [t_1,t^{(1)}]$ and all $\varepsilon_k < \varepsilon^{(1)}$.
By Lemma~\ref{lem:boundedness}, $x_{\varepsilon_k}-p_{\varepsilon_k}$ converges 
to zero a.e.~in $J_T$. Hence there must be some $t^{(2)}\in [t_1,t^{(1)}]$ and 
some $\varepsilon^{(2)}<\varepsilon^{(1)}$ such that $x_{\varepsilon_k}(t^{(2)}) 
< p_{\varepsilon_k}(t_2)-\delta_2$ for some $\delta_2<\delta_1$ and all 
$\varepsilon_k< \varepsilon^{(2)}$. But then because $\dot{x}_{\varepsilon_k}\leq 0$ 
a.e.~in $J$ it follows
\benn
	x_{\varepsilon_k}(t) < p_{\varepsilon_k}(t_2)-\delta_2
\eenn
for all $\varepsilon_k< \varepsilon^{(2)}$ and all $t\in [t^{(2)},t_2]$.
Again because $x_{\varepsilon_k}-p_{\varepsilon_k}$ converges to zero a.e.~in $J_T$ 
this yields
\benn
	\overline{x}(t) = \lim_{k\rightarrow \infty} (p_{\varepsilon_k}(t) 
	- x_{\varepsilon_k}(t) + x_{\varepsilon_k}(t)) < \lim_{k\rightarrow \infty} 
	(p_{\varepsilon_k}(t) - x_{\varepsilon_k}(t) + p_{\varepsilon_k}(t_2) - 
	\delta_2) = \overline{x}(t_2) - \delta_2
\eenn
for all $\varepsilon_k< \varepsilon^{(2)}$ and a.e.~$t\in [t^{(2)},t_2]$. By continuity 
of $\overline{x}$ this estimate holds for all $t\in [t^{(2)},t_2]$ which gives the 
contradiction 
\benn
	\overline{x}(t_2) < \overline{x}(t_2)-\delta_2.
\eenn
Hence, \eqref{eq:coolintgoal} indeed holds also in the limit. From~\eqref{eq:convdomain}
and the previous results we may now conclude that $\dot{\overline{x}}<0$ in a subset of 
$J$, which has positive measure, is only possible if $\overline{x}=F_+(\overline{y})$ 
a.e.~in this set. Note that this precisely gives one case of the variational 
inequality.

Similarly, we can deal with the case $\overline{x}(t_0) = F_-(\overline{y}(t_0))$ to 
show that $\dot{\overline{x}}\geq 0$ and that $\dot{\overline{x}}>0$ in a set of positive 
measure is only possible if $\overline{x} = F_-(\overline{y})$~a.e. in this set. Hence,
it follows, we have for a.e. $t\in J_T$ 
\benn
\dot{\overline{x}}(t)(\overline{x}(t)-\xi) \leq 0\quad 
\forall\ \xi\in [F_-(\overline{y}(t)), F_+(\overline{y}(t))].
\eenn
This proves that $\overline{x}$ solves 
\eqref{Eq:evol_y_1_Lem}-\eqref{Eq:evol_y_3_Lem} with $y=\overline{y}$.
\end{proof}

We observe that the proof of Lemma~\ref{lem:Limit_equations} essentially relied on the
convergence of the fast projection as $\varepsilon\ra 0$. Furthermore, the argument treats 
the critical manifold $\cC_0$ in three different phases according whether we are on the 
boundaries defined by $F_+,F_-$ or in the interior. Since many other hysteresis operators 
can also be defined by variational inequalities, we actually may hope that our strategy
can be carried over to other singular limits not covered by standard normally hyperbolic
Fenichel theory or other fast-slow systems methods. Finally, we can summarize the previous
results and prove the main result. 

\begin{proof}(of Theorem~\ref{thm:projection})
Most of the statements already follow from combining Lemma~\ref{lem:boundedness}, 
Lemma~\ref{lem:convergence} and Lemma~\ref{lem:Limit_equations}. However, we still have 
to prove that $(\overline{x},\overline{y})$ are uniquely determined 
by~\eqref{Eq:evol_x_1}-\eqref{Eq:evol_y_3}.
	
First, we make a few observations. The limit $\overline{x}$ is a generalized play 
operator for the Lipschitz continuous curves $F_+$ and $F_-$ with input $y=\overline{y}$. 
By~\cite[III.2.,Theorem 2.2]{Visintin}, this generalized play is a Lipschitz continuous 
hysteresis operator from $\mathrm{C}(\overline{J_T})\times \mathbb{R}$ to 
$\mathrm{C}(\overline{J_T})$, where the second input variable is the initial condition $x_0$. 
The map $g$ is Lipschitz continuous by \ref{a:fg}. 

To prove uniqueness, we argue by contradiction. Suppose that there are two pairs of 
solutions $(x_1,y_1)$ and $(x_2,y_2)$ of~\eqref{Eq:evol_x_1}-\eqref{Eq:evol_y_3}.
Since $y_1$ and $y_2$ are continuous, for $t$ close enough to $0$, we have
\begin{alignat*}{2}
|y_1(t)-y_2(t)|
&\leq \int_{0}^{t} |g(x_1(s),y_1(s),s) - g(x_2(s),y_2(s),s)| \,\txtd s\\
	&\leq L_g \int_{0}^{t} |x_1(s)-x_2(s)| + |y_1(s)-y_2(s)| \,\txtd s\\
	&\leq c \int_{0}^{t} \sup_{0\leq \tau \leq s}|y_1(\tau)-y_2(\tau)| + |y_1(s)-y_2(s)|  
	\,\txtd s.
\end{alignat*}
We used Lipschitz continuity of the hysteresis operator for the last estimate. Therefore,
\begin{alignat*}{2}
&\sup_{0\leq \tau \leq t}|y_1(\tau)-y_2(\tau)|
\leq c \int_{0}^{t} \sup_{0\leq \tau \leq s}|y_1(\tau)-y_2(\tau)| \,\txtd s,
	\end{alignat*}
and by Gronwall's Lemma it follows $\sup_{0\leq \tau \leq t}|y_1(\tau)-y_2(\tau)|=0$.
This argument can be repeated for another small time interval so that finally $y_1=y_2$.
Consequently, $x_1=x_2$ as well. This shows uniqueness of $\overline{x}$ and $\overline{y}$.
	
Concerning the regularity of $\overline{x}$ and $\overline{y}$ note that	
$\overline{y}\in \mathrm{C}^1(J_T)$ because $g(\overline{x}(\cdot),\overline{y}(\cdot),\cdot)$ 
is continuous. Since $\overline{y}$ is also continuous in $\overline{J_T}$ it follows 
$\overline{y}\in \mathrm{W}^{1,\infty}(J_T)$. We then obtain $\overline{x}\in 
\mathrm{W}^{1,\infty}(J_T)$ by~\cite[III.2.,Theorem 2.3]{Visintin}.
	
The last step is the convergence result, which is relatively simple using the intermediate
results. Indeed, by Lemma~\ref{lem:convergence} and Lemma~\ref{lem:Limit_equations}, 
every sequence $\{\varepsilon_j\}_{j=1}^\I$ with $\varepsilon_j \rightarrow 0$ has a 
subsequence $\{\varepsilon_{j_k}=:\varepsilon_k\}_{k=1}^\I$ such that  
\benn
y_{\varepsilon_k} \rightarrow \overline{y} \text{ in } \mathrm{W}^{1,q}(J_T)
	\quad \text{ and } \quad x_{\varepsilon_k} \rightarrow \overline{x} \text{ in } 
	\mathrm{L}^{q}(J_T).
\eenn
Because $\overline{x}$ and $\overline{y}$ are the unique solution 
of~\eqref{Eq:evol_x_1}-\eqref{Eq:evol_y_3} we conclude that this convergence holds 
for the whole sequence $\{(x_{\varepsilon_j},y_{\varepsilon_j})\}_{j=1}^\I$.	
\end{proof}

Note that the strategy of our proof presented in this section depends crucially 
on the fast variable convergence, which is dealt with via weak convergence first.
Our second approach replaces this step of the argument using a more geometric
strategy based upon local linearization, which actually relies on additional 
assumptions on the differentiability of the vector field; hence, it complements 
the approach presented in this section.

\section{The Main Result II - Linearization Approach}
\label{sec:main2}

As before, we consider the fast-slow non-autonomous planar 
system~\eqref{Eq:evol_epsilon_y}-\eqref{Eq:evol_epsilon_x} on a
finite time interval $J_T=(0,T)$. However, we strengthen
the assumption~\ref{a:fg} to the following: 

\begin{itemize}
		\item[\mylabel{a:fgnew}{(A1')}] $g\in \mathrm{C}^2(\mathbb{R}^3)$ and 
		$f\in \mathrm{C}^2(\mathbb{R}^2)$. 
\end{itemize}

Essentially we are going to get a stronger convergence result if we assume
more differentiability. The result we are going to prove in this section is 
a variant/extension of Theorem~\ref{thm:projection}.

\begin{thm}
\label{Thm:approimation}
Suppose the assumptions~\ref{a:fgnew}, \ref{a:cm}-\ref{a:gbound} hold.
Let $(\overline{x}, \overline{y})$ be the unique solution of
\begin{alignat}{2}
&\dot{y}(t) = g(x(t),y(t),t)\ && \text{a.e. in } J_T,
\label{Eq:evol_with_time_x_1}\\
& y(0) = y_0,\label{Eq:evol_with_time_x_2}\\
&\dot{x}(t)(x(t)-\xi) \leq 0 && \forall\xi\in [F_-(y(t)), F_+(y(t))],
\ \text{a.e. in } J_T,\label{Eq:evol_with_time_y_1}\\
&x(0) =\min\{\max\{F_-(y_0),x_0\},F_+(y_0)\},\label{Eq:evol_with_time_y_2}\\
&x(t)\in [F_-(y(t)), F_+(y(t))] && 
\text{in }\overline{J_T}.\label{Eq:evol_with_time_y_3}\\
\end{alignat}
Then $\overline{y}\in \mathrm{C}(\overline{J_T})\cap\mathrm{C}^1(J_T)$ 
and $\overline{y}\in \mathrm{W}^{1,\infty}(J_T)$. For arbitrary $\eta_1>0$, 
there exists an $\varepsilon_{\eta_1}>0$ such that for all 
$\varepsilon\in (0,\varepsilon_{\eta_1})$ there exists a time 
$t(\varepsilon)\in J_T$ with
\begin{align}
\|y_{\varepsilon} - \overline{y}\|_{\mathrm{C}(\overline{J_T})} 
+ \|x_{\varepsilon} - \overline{x}\|_{ \mathrm{C}[t(\varepsilon),T] } 
<\eta_1.\label{convergence_yeps_hyst}
\end{align}
If $x_0\in [F_-(y_0),F_+(y_0)]$, then $t(\varepsilon)=0$. Otherwise, 
$t(\varepsilon) \leq \varepsilon C(\eta_1)$ for some $C(\eta_1)>0$.
This implies the following:
\begin{itemize}	
	\item[(N1)] If $x_0\in [F_-(y_0),F_+(y_0)]$, then $y_{\varepsilon} 
	\rightarrow \overline{y}$ in $C(\overline{J_T})\cap C^1(J_T)$ 
	and $x_\varepsilon \rightarrow \overline{x}$ in 
	$\mathrm{C}(\overline{J_T})$.	
	\item[(N2)] Otherwise, for arbitrary $\eta_2>0$, $y_{\varepsilon} \rightarrow 
	\overline{y}$ in $C(\overline{J_T})\cap C^1(\eta_2,T)$ and $x_\varepsilon 
	\rightarrow \overline{x}$ in $\mathrm{C}([\eta_2,T])$.
	\item[(N3)] For any $q\in(1,+\I)$, $y_\varepsilon\ra \overline{y}$ in 
	$W^{1,q}(J_T)$ and $x_\varepsilon\ra \overline{x}$ in $L^q(J_T)$.
\end{itemize}
\end{thm} 

In particular, note that the conclusions of convergence to a hysteresis operator
in the singular limit are now obtained in stronger norms in (N2) but the convergence 
result of Theorem~\ref{thm:projection} stated in (N3) obviously still holds. We do not expect
any stronger convergence in (N2) for the fast variable, even if the differentiability 
of $f,g$ is increased. This is reminiscent of the classical results from Fenichel
Theory~\cite{Fenichel4,Jones,KuehnBook} as fast subsystem trajectories 
generically develop non-differentiable points when connecting to a critical 
manifold. To prove~Theorem~\ref{Thm:approimation}, we need some additional
notation. First, note that by Lemma~\ref{lem:boundedness}, trajectories remain 
bounded.

\begin{defn}
Let $M$ be a compact rectangle such that $(x_\varepsilon(t),y_\varepsilon(t))$ 
is contained in $M$ for all $\varepsilon>0$ and all $t\in \overline{J_T}$. 
	\begin{itemize}
	\item	We introduce $M_0:=M\cap \{(x,y):f(x,y)=0 \}$, $M_+:=M\cap\{(x,y):f(x,y)<0\}$ 
	and $M_-:= M\cap \{(x,y):f(x,y)>0 \}$.
	\item We define the constants $C_f,C_g,C_{Dg},C_{Df},C_{D^2f},C_{D^2g}>0$ by 
	the upper bounds of the norms of $f,g,Dg,Df,D^2f$ and $D^2g$ on $M$. Moreover, we set 
	$C_M:=\max\{\|w\|: w\in M \}$, $C_{D^2}:=\max\{C_{D^2f},C_{D^2g}\}$ and 
	$L:=\max\{L_f,L_g\}$; cf.~assumptions~\ref{a:fg} and \ref{a:fgnew}.
	\end{itemize}
\end{defn}

We remark that the notation of $M_+$ and $M_-$ corresponds to $F_+$ and $F_-$ above and
the sign subscripts are chosen so that $[F_-,F_+]$ is an interval.

\begin{defn}
\label{Def:exact_solution_on_J_T}
For $w_0=(x_{w_0},y_{w_0})\in M$ and $\varepsilon>0$ we write 
$x_{\varepsilon,w_0}$ and $y_{\varepsilon,w_0}$ for the solution 
of~\eqref{Eq:evol_epsilon_y}-\eqref{Eq:evol_epsilon_x} 
with initial value $(x(0),y(0))=w_0$ and on the time interval $J_T$.
For $(\tau_0,\tau_1)\subset J_T$ we denote by $x_{\varepsilon,
w_0,(\tau_0,\tau_1)}$ and $y_{\varepsilon,w_0,(\tau_0,\tau_1)}$
the solution of~\eqref{Eq:evol_epsilon_y}-\eqref{Eq:evol_epsilon_x}
with initial value $(x(\tau_0),y(\tau_0))=w_0$ and on the time 
interval $(\tau_0,\tau_1)$.
\end{defn}

\begin{figure}[htbp]
	\centering
\begin{overpic}[width=0.3\linewidth]{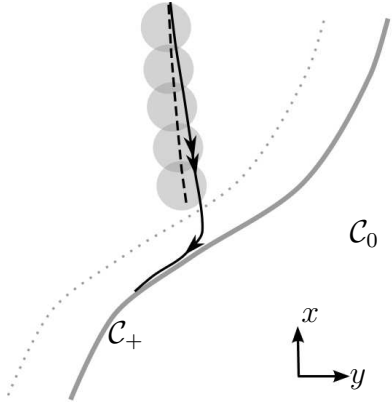}
\put(85,5){$y$}
\put(73,20){$x$}
\put(85,40){$\cC_0$}
\put(25,15){$\cC_+$}
\end{overpic}
\caption{\label{fig:09}Sketch of the approximation argument. The actual
fast-slow system trajectory from Definition~\ref{Def:exact_solution_on_J_T}
is shown as a solid curve (black, with arrows). On the fast scale we use
local approximation in balls (grey disks) to obtain the curve in 
Definition~\ref{Def:linearized_solution} (dashed black). The 
critical manifold boundary $\cC_+$ is shown as well (thick grey curve).
Furthermore, we always work with the approximation up to a given 
neighbourhood, which scales with $\delta$ as in 
Definition~\ref{Lem:estim:time_linearized sol}; the neighbourhood is
indicated as well (dotted grey curve).}
\end{figure}

The additional subscript $w_0$ will be necessary in the proof as we piece
together several local results comparing linear and nonlinear terms. This 
approach is similar to a strategy of patching sample
paths used in the context of stochastic fast-slow 
systems~\cite{BerglundGentz,BerglundGentzKuehn1}; see also Figure~\ref{fig:09}. 
With $w\in M$, $t_0\in \overline{J_T}$ and $\varepsilon>0$ consider the following 
linearized system of evolution equations:
\begin{alignat}{2}
\varepsilon \dot{x}(t)& = f(w) + [\mathrm{D}f(w)]\left(
\left(\begin{matrix} x(t)\\ y(t) \end{matrix}\right) - w\right)
\ && \text{ for }t>t_0,\label{Eq:evol_lin_y_1}\\
x(t_0) &= x_w\label{Eq:evol_lin_y_2},\\
\dot{y}(t) &= g(w,t_0) + [\partial_{(x,y)}g(w,t_0)]\left(
\left(\begin{matrix} x(t)\\ y(t) \end{matrix}\right)- w
\right)\notag \\&\ + [\partial_t g(w,t_0)](t-t_0)
\ && \text{ for }t>t_0,\label{Eq:evol_lin_x_1}\\
y(t_0) &= y_w.\label{Eq:evol_lin_x_2}
\end{alignat}
We group the terms containing $(x(t),y(t))$ and the remaining terms together, write
\begin{alignat*}{2}
F_1(w) + [F_2(w)]
\left(\begin{matrix}x(t)\\y(t)\end{matrix}\right)
&:=f(w) + [\mathrm{D}f(w)]\left(
\left(\begin{matrix}x(t)\\ y(t) \end{matrix}\right)- w\right),\\
G_1(w,t_0,t) + [G_2(w,t_0)]
\left(\begin{matrix} x(t)\\y(t) \end{matrix}\right)&:=g(w,t_0) 
+ [\partial_{(x,y)}g(w,t_0)]\left(
\left(\begin{matrix} x(t)\\ y(t)\end{matrix}\right)- w\right)\\
&\quad + [\partial_t g(w,t_0)](t-t_0),
\end{alignat*}
and denote by $C_{F_1},C_{F_2},C_{G_1},C_{G_2}>0$ upper bounds of the maximum 
norms of the corresponding functions $F_j, G_j$ with $j\in\{1,2\}$ in $M$ and 
$\overline{J_T}$. We also set 
\benn
C_{F,G}:=\max\{C_{F_1},C_{G_1},C_{F_2},C_{G_2}\}.
\eenn
The next step is to define the local approximations of the
solution and patch them together on a given finite-time interval.

\begin{defn}
\label{Def:linearized_solution}
For $v=(x_v,y_v)\in M$, for given $\varepsilon,\theta>0$ and $(\tau_0,\tau_1)
\subset J_T$, we define $x^{\mathrm{lin}}_{\varepsilon,\theta,v,(\tau_0,\tau_1)}$ and 
$y^{\mathrm{lin}}_{\varepsilon,\theta,v,(\tau_0,\tau_1)}$ in the following way:
\begin{enumerate}
	\item[(S1)] Let $v^{(0)}:=v$, $\tau^{(0)}:=\tau_0$ and define 
	$x^{\mathrm{lin}}_{\varepsilon,\theta,v,(\tau_0,\tau_1)}$ and 
	$y^{\mathrm{lin}}_{\varepsilon,\theta,v,(\tau_0,\tau_1)}$ first on the interval 
	$[\tau^{(0)},\tau^{(1)})$ as the solution 
	of~\eqref{Eq:evol_lin_y_1}-\eqref{Eq:evol_lin_x_2} with initial value 
	$w=v^{(0)}$ and initial time $\tau^{(0)}$. Let $\tau^{(1)}$ be the first time in 
	$(\tau_0,\tau_1]$ with 
	\be
	\label{eq:hit1}
	|v^{(0)}_x-x^{\mathrm{lin}}_{\varepsilon,\theta,v,(\tau_0,\tau_1)}(\tau^{1})| 
	= \frac{\theta}{2}\quad \text{ or }\quad \tau^{(1)}=\tau_1.
	\ee
	We set $v^{(1)}:= (x^{\mathrm{lin}}_{\varepsilon,\theta,v,(\tau_0,\tau_1)}(\tau^{1}),
	y^{\mathrm{lin}}_{\varepsilon,\theta,v,(\tau_0,\tau_1)}(\tau^{1})).$
	\item[(S2)] For given $v^{(i)}\in M$ and $\tau^{(i)}\in (\tau_0,\tau_1)$, 
	$i\geq 1$, we define $x^{\mathrm{lin}}_{\varepsilon,\theta,v,(\tau_0,\tau_1)}$ 
	and $y^{\mathrm{lin}}_{\varepsilon,\theta,v,(\tau_0,\tau_1)}$ inductively on the 
	interval $[\tau^{(i)},\tau^{(i+1)})$ by the solution 
	of~\eqref{Eq:evol_lin_y_1}-\eqref{Eq:evol_lin_x_2} with initial value 
	$w=v^{(i)}$ and initial time $\tau^{(i)}$. $\tau^{(i+1)}$ is the first time in 
	$(\tau^{(i)},\tau_1]$ with 
	\be
	\label{eq:hit2}
	|v^{(i)}_x-x^{\mathrm{lin}}_{\varepsilon,\theta,v,(\tau_0,\tau_1)}(\tau^{i+1})| 
	= \frac{\theta}{2}\quad \text{ or }\quad \tau^{(i+1)}=\tau_1.
	\ee
	We set $v^{(i+1)}:= (x^{\mathrm{lin}}_{\varepsilon,\theta,v,(\tau_0,\tau_1)}(\tau^{i+1}),
	y^{\mathrm{lin}}_{\varepsilon,\theta,v,(\tau_0,\tau_1)}(\tau^{i+1})).$ If 
	$\tau^{(i+1)}<\tau_1$ we repeat Step (S2).
	\item[(S3)] After finitely many steps this defines 
	$x^{\mathrm{lin}}_{\varepsilon,\theta,v,(\tau_0,\tau_1)}$ and 
	$y^{\mathrm{lin}}_{\varepsilon,\theta,v,(\tau_0,\tau_1)}$ on the 
	interval $[\tau_0,\tau_1]$.
	\end{enumerate}
\end{defn}

Although Definition~\ref{Def:linearized_solution} may look complicated, it is
actually just a piecewise definition of the linearized solution using hitting
times in~\eqref{eq:hit1}-\eqref{eq:hit2}.

\begin{defn}
Let $\delta>0$ be given. Then by assumption~\ref{a:cm} we can 
define $f_+(\delta)<0$ by 
\benn
f_+(\delta):=\frac{1}{2}\max\{f(x,y): (x,y)\in M_+, 
\mathrm{dist}((x,y),M_0)\geq \delta\}.
\eenn
Similarly we define $f_-(\delta)>0$ by 
\benn
f_-(\delta):=\frac{1}{2}\min\{f(x,y): (x,y)\in M_-, 
\mathrm{dist}((x,y),M_0)\geq \delta\}.
\eenn
Furthermore, it will be helpful to introduce the following notation
\be
\min\{|f_+(\delta)|,f_-(\delta)\}=:f_m(\delta).
\ee
\end{defn} 

The next lemma is the main step to estimate the deviation of the solution
obtained from the linearized process by patching to the true solution.
To simplify the statement and the proof, we use for the
next result the notation $x_{\varepsilon,w_0}:=x_{w_0}$, 
$y_{\varepsilon,w_0}:=y_{w_0}$, $x^{\mathrm{lin}}:=
x^{\mathrm{lin}}_{\varepsilon,\theta,v,(\tau_0,\tau_1)}$ and 
$y^{\mathrm{lin}}:=y^{\mathrm{lin}}_{\varepsilon,\theta,v,(\tau_0,\tau_1)}$.

\begin{lem}
\label{Lem:estim:time_linearized sol}
Let $\delta>0$ be given such that $M\backslash \overline{(M_0+B(0,\delta))}
\neq \emptyset$ and let $v=(x_v,y_v)\in M\backslash \overline{(M_0+B(0,\delta))}$.
Furthermore, consider any $\tau_0\in [0,T)$ and define 
\benn
\varepsilon_\delta := 
\min\left\{\frac{f_m(\delta)}{2(1+L_\pm)C_{F,G}}
\left[\left(\sqrt{2} \left(C_M+ \frac{2C_{F,G}}{f_m(\delta)}
\right)\txte^{2C_{F,G}/f_m(\delta)}+1\right)\right]^{-1},1\right\}.
\eenn
Let $\varepsilon\in (0,\varepsilon_\delta)$ and 
$\theta \in \left(0, \min\{\frac{f_m(\delta)}{C_{Df}}, 1\} \right)$ be 
arbitrary but fixed. Denote by $\tau_1>\tau_0$ the first hitting time  
such that $(x^{\mathrm{lin}},y^{\mathrm{lin}})\in \overline{M_0+B(0,\delta)}$ or 
$\tau_1=T$. We have have the local time estimate 
\be
\label{eq:tauestlem2}
|\tau^{(j+1)}-\tau^{(j)}|\leq 
\frac{\varepsilon\theta}{f_m(\delta)},\quad \text{for all }~ 
j\in \{0,\cdots, K-1\},
\ee
and the global number $K$ of time intervals $(\tau^{(j)},\tau^{(j+1)})$ 
in Definition~\ref{Def:linearized_solution} satisfies
\be
\label{eq:keyKbound}
K\leq  \left\lceil\frac{2~(1+L_\pm)~\mathrm{dist}(v, M_0+B(0,\delta))}{\theta
\left(1-\frac{\varepsilon}{\varepsilon_\delta} \right)}\right\rceil \leq 
\frac{2~(1+L_\pm)~\mathrm{dist}(v, M_0+B(0,\delta))}{\theta\left(1-
\frac{\varepsilon}{\varepsilon_\delta}\right)} + 1.
\ee
In particular, this implies the global time estimate
\be
\label{eq:test01}
|\tau_1-\tau_0| \leq \varepsilon \left(  
\frac{2~(1+L_\pm)~\mathrm{dist}(v, M_0+B(0,\delta))}{f_m(\delta)
\left(1-\frac{\varepsilon}{\varepsilon_\delta}\right)}+ \frac{1}{C_{Df}}\right).
\ee
Moreover, for $w_0\in M$ and arbitrary $t\in [\tau_0,\tau_1]$ there holds
\be
\label{eq:solestLem}
	| x_{w_0}(t) - x^{\mathrm{lin}}(t)|
	+ | y_{w_0}(t) - y^{\mathrm{lin}}(t)|
	\leq 
\left(| x_{w_0}(\tau_0) - x_v| + | y_{w_0}(\tau_0) - y_v|
	+ \theta^2 \cK_1\right)\cK_2,
\ee	
where the constants $\cK_1,\cK_2$ depend upon the given data as follows
\beann
\cK_1&=&\left(\frac{2~(1+L_\pm)~\mathrm{dist}(v, M_0+B(0,\delta))}{f_m(\delta)\left(1-
\frac{\varepsilon}{\varepsilon_\delta}\right)}+ \frac{1}{C_{Df}}\right)
	\left(2~C_{D^2} + \frac{\varepsilon^3}{f_m(\delta)^2}\right),\\
	\cK_2&=&\exp\left[2L \left(  
	\frac{2~(1+L_\pm)~\mathrm{dist}(v, M_0+B(0,\delta))}{f_m(\delta)
	\left(1-\frac{\varepsilon}{\varepsilon_\delta}\right)}+ \frac{1}{C_{Df}}\right)\right]
\eeann
\end{lem}

\begin{proof}
For $j\in \{0,\cdots, K-1\}$ let $(\tau^{(j)},\tau^{(j+1)})$ be an interval 
as in Definition~\ref{Def:linearized_solution}. On this interval, we have
\begin{align}
\left(\begin{matrix}x^{\mathrm{lin}}(t)\\ 
y^{\mathrm{lin}}(t) \end{matrix} \right)
&= \exp\left[(t-\tau^{(j)}) \left(\begin{matrix}
\frac{1}{\varepsilon}F_2(v^{(j)})\\ G_2(v^{(j)},\tau^{(j)})\end{matrix}\right)
\right]v^{(j)}\notag\\
&+\int_{\tau^{(j)}}^{t}\exp\left[(t-s) \left(
\begin{matrix} \frac{1}{\varepsilon}F_2(v^{(j)})\\ G_2(v^{(j)},\tau^{(j)})
\end{matrix}\right)\right]\left(\begin{matrix}\frac{1}{\varepsilon}F_1(v^{(j)},s)\\
G_1(v^{(j)})\end{matrix}\right)\,\txtd s.\label{solution_lin_syst}
\end{align}
Consider the matrix
\[
\left(\begin{matrix}
[F_2(v^{(j)})]^{(1)} && [F_2(v^{(j)})]^{(2)}\\
[G_2(v^{(j)},\tau^{(j)})]^{(1)} && [G_2(v^{(j)},\tau^{(j)})]^{(2)}
\end{matrix}
\right):=\left(
\begin{matrix}
F_2(v^{(j)})\\
G_2(v^{(j)},\tau^{(j)})
\end{matrix}\right).
\]
Note that the entries of the matrix are bounded by $C_{F,G}$. By definition of 
the matrix exponential it follows that for all $s,t\in [\tau^{(j)},\tau^{(j+1)}]$ 
\begin{align*}
(0,1)\left[\exp\left[(t-s)\left(\begin{matrix}\frac{1}{\varepsilon}
F_2(v^{(j)})\\ G_2(v^{(j)},\tau^{(j)}) \end{matrix}\right)\right]\right]
= ( m_1(t,s),1+ m_2(t,s)),
\end{align*}
where for $l\in \{1,2\}$ a direct calculation yields
\begin{align*}
|m_l(t,s)| &\leq \sum_{k=1}^\infty \frac{((t-s)C_{F,G})^k}{k!}\left(1+\frac1\varepsilon\right)^{k-1} 
= C_{F,G}(t-s) \sum_{k=0}^\infty \left(\frac{(t-s)C_{F,G}(1+\varepsilon) }{\varepsilon}\right)^k
\frac{1}{(k+1)!} \\
&\leq C_{F,G}(t-s)\exp\left(\frac{2C_{F,G}(t-s)}{\varepsilon}\right) 
\leq C_{F,G}(\tau^{(j+1)}-\tau^{(j)})\exp
\left(\frac{2C_{F,G}(\tau^{(j+1)}-\tau^{(j)})}{\varepsilon}\right).
\end{align*}
Consequently, we also deduce that
\benn
(|m_1(t,s)|^2 + |m_2(t,s)|^2)^{1/2} \leq 
\sqrt{2}~C_{F,G}(\tau^{(j+1)}-\tau^{(j)})
\exp\left(\frac{2C_{F,G}(\tau^{(j+1)}-\tau^{(j)})}{\varepsilon}\right).
\eenn
We multiply with $(0,1)$ from the left in~\eqref{solution_lin_syst} for 
$j\in \{0,\cdots, K-1\}$ and $t\in [\tau^{(j)},\tau^{(j+1)}]$ and take the absolute 
value to obtain
\beann
|y^{\textnormal{lin}}(t)-	y_{v^{(j)}}|
&=& \left|(m_1(t,\tau^{(j)}), m_2(t,\tau^{(i)}))v^{(j)} + \int_{\tau^{(j)}}^{t}
(m_1(t,s), 1+m_2(t,s)) \left(
\begin{matrix}
\frac{1}{\varepsilon}F_1(v^{(j)})\\
G_1(v^{(j)},s)
\end{matrix}\right)
\,\txtd s\right| \\
&\leq& \sqrt{2}~C_{F,G}(\tau^{(j+1)}-\tau^{(j)})
~\txte^{\frac{2C_{F,G}(\tau^{(j+1)}-\tau^{(j)})}{\varepsilon}}
\left(C_M + (\tau^{(j+1)}-\tau^{(j)}) C_{F,G}\left(1+ \frac{1}{\varepsilon}
\right)\right)\\&&+(\tau^{(j+1)}-\tau^{(j)})C_{F,G}.
\eeann
Since $\theta \in (0,1]$, by the definition of $\varepsilon_\delta$, 
using $\varepsilon \in (0,\varepsilon_\delta)$ and if~\eqref{eq:tauestlem2} 
would hold, i.e., $|\tau^{(j+1)}-
\tau^{(j)}|\leq \frac{\varepsilon\theta}{f_m(\delta)}$, then it follows that
\beann
|y^{\textnormal{lin}}(t)-y_{v^{(j)}}| &\leq & 
\frac{\theta\varepsilon C_{F,G}}{f_m(\delta)}\left(\sqrt{2}~ 
\txte^{2C_{F,G}/f_m(\delta)}\left(C_M + \frac{2C_{F,G}}{f_m(\delta)}
\right)+1\right)\\
&\leq& \frac{\theta\varepsilon}{2(1+L_\pm)\varepsilon_\delta}
\leq \frac{\theta}{2}.
\eeann
Next, we are going to prove~\eqref{eq:tauestlem2}. We already know that 
if $|\tau^{(j+1)}-\tau^{(j)}|\leq 
\frac{\varepsilon\theta}{f_m(\delta)}$ then 
$(x^{\mathrm{lin}}(t),y^{\mathrm{lin}}(t))\in B(v^{(j)},\theta)$ 
for all $t\in [\tau^{(j)},\tau^{(j+1)}]$. Recall that $\tau_1$ was defined 
such that $(x^{\mathrm{lin}}(t),y^{\mathrm{lin}}(t))$ lies outside of
a $\delta$-neighbourhood of the critical manifold, i.e., in 
$M\backslash (M_0+B(0,\delta))$. Assume that $v\in M_+$. Then we can obtain 
a local upper bound on the fast vector field of the following form
\be
f(v^{(j)}) + [\mathrm{D}f(v^{(j)})]\left(
\left(\begin{matrix}x^{\mathrm{lin}}(t)\\
y^{\mathrm{lin}}(t)\end{matrix}\right)
- v^{(j)}\right)  < 2f_+(\delta) + C_{Df}\theta < f_+(\delta)
\label{estimate_g_in_Mplus}
\ee
for all $t\in [\tau^{(j)},\tau^{(j+1)}]$. This actually implies a helpful
bound at the end of the small time interval, namely 
\bea
x^{\mathrm{lin}}(\tau^{(j+1)}) &=& v^{(j)}_x 
+ \int_{\tau^{(j)}}^{\tau^{(j+1)}} \frac{1}{\varepsilon}
\left(	F_1(v^{(j)}) + F_2(v^{(j)})\left[
\left(\begin{matrix}(x^{\mathrm{lin}})(s)\\(y^{\mathrm{lin}})(s)
\end{matrix}\right) - v^{(j)}\right]\right)\,\txtd s\nonumber\\
&\leq& v^{(j)}_x + \frac{(\tau^{(j+1)}-\tau^{(j)})}{\varepsilon}f_+(\delta)
\leq v^{(j)}_x - \frac{(\tau^{(j+1)}-\tau^{(j)})}{\varepsilon}f_m(\delta).
\label{order_convergence_ylin}
\eea
The corresponding result for $v\in M_-$ follows analogously.
By definition of the patched linear solution, the result~\eqref{eq:tauestlem2}
follows. We have collected enough estimates to derive the upper 
bound~\eqref{eq:keyKbound} for $K$. Suppose first that $v\in M_+$. Then for 
all $j\in \{0,\cdots, K-1\}$ there holds $x_{v^{(j+1)}} \leq x_{v^{(j)}}$ and 
\benn
\left|y_{v^{(j)}} - y_{v^{(i+1)}}\right| \leq  
\frac{\theta\varepsilon}{2(1+L\pm)\varepsilon_\delta}.
\eenn
By assumption~\ref{a:Fpm}, we have that $F_+$ is monotone increasing with 
Lipschitz constant $L_+$. Therefore
for all $j\in \{0,\cdots, K-1\}$:
\beann
\mathrm{dist}(v^{(j)}, M_0+B(0,\delta))
&\leq &(1+L_\pm)\mathrm{dist}(v, M_0+B(0,\delta)) + \sum_{k=0}^{j-1} 
\left(L_+|y_{v^{(k+1)}} - y_{v^{(k)}}| -\frac{\theta}{2}\right)\\
&\leq& (1+L_\pm)\mathrm{dist}(v, M_0+B(0,\delta)) + j\frac{\theta}{2}
\left(\frac{\varepsilon}{\varepsilon_\delta}-1\right).
\eeann
Hence, $v^{(K)}\in \overline{M_0+B(0,\delta)}$ if $K$ is large enough, 
so that
\benn
(1+L_\pm)\mathrm{dist}(v, M_0+B(0,\delta)) + K\frac{\theta}{2}
\left(\frac{\varepsilon}{\varepsilon_\delta}-1\right)\leq 0,
\eenn
which can be attained for some
\benn
K\leq  \left\lceil\frac{2(1+L_\pm)\mathrm{dist}(v, M_0+B(0,\delta))}{\theta
\left(1-\frac{\varepsilon}{\varepsilon_\delta}\right)}\right\rceil 
\leq \frac{2(1+L_\pm)\mathrm{dist}(v, M_0+B(0,\delta))}{\theta
\left(1-\frac{\varepsilon}{\varepsilon_\delta}\right)} + 1
\eenn
and~\eqref{eq:keyKbound} follows as analogous estimates for $v\in M_-$ 
lead to the same bounds for $K$; note that in this case we have 
$x_{v^{(j+1)}} \geq x_{v^{(j)}}$ for $j\in \{0,\cdots, K-1\}$. 
The upper bound for $K$ implies almost immediately 
the total time estimate~\eqref{eq:test01}. For example, consider the 
direct calculation in the case of $v\in M_+$ 
\beann
|\tau_1-\tau_0| &\leq& K\frac{\varepsilon\theta}{f_m(\delta)} 
\leq \left(\frac{2(1+L_\pm)\mathrm{dist}(v, M_0+B(0,\delta))}{\theta
\left(1-\frac{\varepsilon}{\varepsilon_\delta}\right)} + 1\right) 
\frac{\varepsilon\theta}{f_m(\delta)}\\
&\leq& \varepsilon \left( \frac{2(1+L_\pm)\mathrm{dist}(v, M_0+B(0,
\delta))}{f_m(\delta)\left(1-\frac{\varepsilon}{\varepsilon_\delta}\right)}
+ \frac{1}{C_{Df}}\right).
\eeann
With the bounds on time intervals, one may now inductively show the 
the worst-case upper bound~\eqref{eq:solestLem} by a second-order 
approximation of $f$ and $g$ in combination with Gronwall's Lemma and 
$\varepsilon<\varepsilon_\delta\leq 1$. 	
\end{proof}

We are also going to need a preliminary estimate for the full nonlinear
solution near the boundary of the critical manifold.

\begin{lem}
\label{Lem:estim_time_exact_sol}
Let $\delta,\varepsilon>0$ and $v=(x_v,y_v)\in \overline{M_0+B(0,\delta)}$ be given.
Consider any $\tau_0\subset [0,T)$ and denote by $\tau_1$ the first time 
after $\tau_0$ such that 
$(x_{\varepsilon,v,(\tau_0,\tau_1)},y_{\varepsilon,v,(\tau_0,\tau_1)})=:(x,y)\in 
\partial(M_0+B(0,2\delta))$ or $\tau_1=T$.
Then either 
\be
\label{eq:firstfullbd}
\tau_1=T \qquad \text{ or }\qquad |\tau_1-\tau_0|>\frac{\delta}{C_g}.
\ee
Furthermore, for arbitrary $w_0\in M$, $(x_{w_0},y_{w_0}):=(x_{\varepsilon,w_0},
y_{\varepsilon,w_0})$ and $t\in (\tau_0,\tau_1)$ there 
holds
\benn
| x_{w_0}(t) - x(t)|+ | y_{w_0}(t) - y(t)|\leq\left(|x_{w_0}(\tau_0)-x_v|
+ | y_{w_0}(\tau_0) - y_v|\right)
	 \txte^{ (\tau_1-\tau_0)L \left(1+1/\varepsilon\right)}.
\eenn
\end{lem}

\begin{proof}
We consider the case 
$v\in M_+$. Then because $f(x,y)<0$ for all $(x,y)\in M_+$ and because 
$F_+$ is monotone increasing, the closest point in $\partial(M_0+B(0,2\delta))$, 
which can be reached from $v$ is given by 
\benn
(x_v,y_\alpha):=\{(x_v,y_v-\alpha): \alpha >0\}\cap 
\partial(M_0+B(0,2\delta)),
\eenn
and $y_v - y > y_v - y_\alpha$ for all points $(x,y) \in (M_0 + B(0,2\delta))$ 
with $x \leq x_v$. 
Since $v\in (M_0+B(0,\delta))$ we have $y_v-y_\alpha> \delta$. Hence, 
either $\tau_1=T$ or $|y(\tau_1)-y_v| >\delta$. This yields
\benn
\delta< |y(\tau_1)-y_v |= 
\left|\int_{\tau_0}^{\tau_1} \dot{y}(s)\, \txtd s\right| 
= \left|\int_{\tau_0}^{\tau_1} g(x(s),y(s),s)\, \txtd s \right|\leq C_g 
(\tau_1-\tau_0).
\eenn 
Therefore, we find $|\tau_1-\tau_0|>\frac{\delta}{C_g}$. Analogous estimates 
for $v\in M_-$ prove~\eqref{eq:firstfullbd}. The last statement in the result
follows by a direct Gronwall Lemma argument.
\end{proof}

It is helpful to describe the solution near the critical manifold by the 
full nonlinear dynamics as it reduces to the slow dynamics in the singular
limit, while still using the patched linearized solution for the fast 
dynamics. This motivates the following definition:

\begin{defn}
\label{Def:approximate_solution}
Let $w_0=(x_0,y_0)\in M$ and $\delta>0$ be given. Adopt the assumptions and 
the notation from Lemma~\ref{Lem:estim:time_linearized sol}. Let 
$\varepsilon\in \left(0,\frac{\varepsilon_\delta}{2}\right)$ be arbitrary.
Furthermore, choose $\theta_\varepsilon \in 
\left(0, \min\{\frac{f_m(\delta)}{C_{Df}}, 1\} \right)$ such that 
$\theta_\varepsilon$ is of order $o\left(\txte^{-\frac{TL}{2\varepsilon}}\right)$.
We define $\tilde{x}_{\varepsilon}$ and $\tilde{y}_{\varepsilon}$ inductively as
follows:
\begin{enumerate}
		\item[(T1)] Let $t_0:=0$ be the initial time. If $\mathrm{dist}(w_0,M_0)\leq \delta$, define 
		$t_1>t_0$ to be the first time such that $(x_{\varepsilon,w_0,(t_0,t_1)},
		y_{\varepsilon,w_0,(t_0,t_1)})\in \partial(M_0+B(0,2\delta))$ or $t_1=T$. In 
		this case we set
		\benn
		(\tilde{x}_{\varepsilon},\tilde{y}_{\varepsilon}):= 
		(x_{\varepsilon,w_0,(t_0,t_1)},y_{\varepsilon,w_0,(t_0,t_1)}) \quad \text{
		on $[t_0,t_1]$}.
		\eenn
		If $\mathrm{dist}(w_0,M_0)> \delta$ we define $t_1>t_0$ by the first time 
		such that the linearized solution satisfies
		$(x^{\mathrm{lin}}_{\varepsilon,\theta,w_0,(t_0,t_1)},
		y^{\mathrm{lin}}_{\varepsilon,\theta,w_0,(t_0,t_1)})\in 
		\overline{(M_0+B(0,\delta))}$ or $t_1=T$. We then set 
		\benn
		(\tilde{x}_{\varepsilon},\tilde{y}_{\varepsilon}):= 
		(x^{\mathrm{lin}}_{\varepsilon,\theta,w_0,(t_0,t_1)},
		y^{\mathrm{lin}}_{\varepsilon,\theta,w_0,(t_0,t_1)}) \quad \text{
		on $[t_0,t_1]$}.
		\eenn
		In both cases we define $w_1:=(\tilde{x}_{\varepsilon}(t_1),
		\tilde{y}_{\varepsilon}(t_1))$.
		\item[(T2)] Let $(\tilde{x}_{\varepsilon},\tilde{y}_{\varepsilon})$
		be defined on the interval $\overline{J_{t_i}}$ and let $w_0,\ldots, w_i$ be 
		chosen. If $t_i=T$ we are done. Otherwise, $t_i<T$. 
		If $w_i\in \partial(M_0+B(0,\delta))$, define $t_{i+1}>t_i$ by the first 
		time such that $(x_{\varepsilon,w_i,(t_i,t_{i+1})},
		y_{\varepsilon,w_i,(t_i,t_{i+1})})\in \partial(M_0+B(0,2\delta))$ or $t_{i+1}=T$.
		In this case we set
		\benn
		(\tilde{x}_{\varepsilon},\tilde{y}_{\varepsilon}):= 
		(x_{\varepsilon,w_i,(t_i,t_{i+1})},y_{\varepsilon,w_i,(t_i,t_{i+1})})\quad \text{ 
		on $[t_i,t_{i+1}]$}.
		\eenn
		If $w_i\in \partial(M_0+B(0,2\delta))$, define $t_{i+1}>t_i$ by the first time 
		such that the linearized solution satisfies 
		$(x^{\mathrm{lin}}_{\varepsilon,\theta,w_i,(t_i,t_{i+1})},
		y^{\mathrm{lin}}_{\varepsilon,\theta,w_i,(t_i,t_{i+1})})\in 
		\partial(M_0+B(0,\delta))$ or $t_{i+1}=T$. In this case we set
		\benn
		(\tilde{x}_{\varepsilon},\tilde{y}_{\varepsilon}):= 
		(x^{\mathrm{lin}}_{\varepsilon,\theta,w_i,(t_i,t_{i+1})},
		y^{\mathrm{lin}}_{\varepsilon,\theta,w_i,(t_i,t_{i+1})})\quad\text{ 
		on $[t_i,t_{i+1}]$}.
		\eenn
		In both cases we define 
		$w_{i+1}:=(\tilde{x}_{\varepsilon}(t_{i+1}),\tilde{y}_{\varepsilon}(t_{i+1}))$.
	\end{enumerate}
\end{defn}

As before, we need a projection operator for the solution. Define  
$\tilde{p}_{\varepsilon}=p(\tilde{x}_\varepsilon,\tilde{y}_\varepsilon)$ by
\be
\label{eq:deftildeproject}
\tilde{p}_{\varepsilon}:= p(\tilde{x}_\varepsilon,\tilde{y}_\varepsilon) 
= \min\{\max\{\tilde{x}_\varepsilon, F_-(\tilde{y}_\varepsilon)\},
F_+(\tilde{y}_\varepsilon)\}.
\ee
		
\begin{lem}
\label{Lem:estim:time_approx_sol}
Consider the same assumptions and the notation 
from Definition~\ref{Def:approximate_solution}. Since $w_0\in M$ is fixed
we denote $x_{\varepsilon,w_0}$ by $x_\varepsilon$ and $y_{\varepsilon,w_0}$ 
by $y_\varepsilon$. Then there exists a constant $C(\delta)>0$ such that
\be
\label{eq:lemtilde1}
| x_\varepsilon(t) - \tilde{x}_\varepsilon(t)|
+ | y_\varepsilon(t) - \tilde{y}_\varepsilon(t)|\leq 
\theta_\varepsilon^2\txte^{\frac{TL}{\varepsilon}}C(\delta),\quad
\text{$\forall t\in \overline{J_T}$.}
\ee
Since $\theta_\varepsilon$ is of order 
$o\left(\txte^{-\frac{TL}{2\varepsilon}}\right)$, this implies that
\benn
\|(x_\varepsilon,y_\varepsilon) - 
(\tilde{x}_\varepsilon,\tilde{y}_\varepsilon)\|_{C(\overline{J_T};\mathbb{R}^2)} 
\leq \theta_\varepsilon^2\txte^{\frac{TL}{\varepsilon}}C(\delta)\rightarrow 0
\eenn
as $\varepsilon\rightarrow 0$. Furthermore, if 
$w_0=(x_{w_0},y_{w_0})\in M_0+B(0,\delta)$ then 
\be
\label{eq:lemtilde2}
\|\tilde{p}_\varepsilon - \tilde{x}_\varepsilon\|_{C(\overline{J_T})} 
\leq 2\delta(1+L_\pm).
\ee
If $w_0=(x_{w_0},y_{w_0})\in M\backslash \overline{(M_0+B(0,\delta))}$ 
then there exists a constant $C_0(\delta)>0$ and a time $t_1\in J_T$ with
$t_1 \leq \varepsilon C_0(\delta)$ such that
\be
\label{eq:lemtilde3}
\|\tilde{p}_\varepsilon - \tilde{x}_\varepsilon\|_{C[t_1,T]} 
\leq 2(1+L_\pm)\delta.
\ee
\end{lem}

\begin{proof}
During the proof we denote $\tilde{x}:=\tilde{x}_{\varepsilon}$, 
$\tilde{y}:=\tilde{y}_{\varepsilon}$, $x:=x_\varepsilon$, and 
$y:=y_\varepsilon$. By Lemma~\ref{Lem:estim_time_exact_sol}, 
each subinterval of $J_T$ in which $(\tilde{x},\tilde{y})$ behaves according 
to~\eqref{Eq:evol_epsilon_y}-\eqref{Eq:evol_epsilon_x} can be estimated from 
below by $\delta/C_g$. Hence, the total number $K$ of subintervals 
in Definition~\ref{Def:approximate_solution} is bounded from above by 
$(TC_g)/\delta$. Let $K=K_{l}+K_{e}$, where $K_{l}$ is the number of 
time intervals in which $(\tilde{x},\tilde{y})$ is defined 
via~\eqref{Eq:evol_lin_y_1}-\eqref{Eq:evol_lin_x_2}, and where $K_{e}$ 
denotes the number of time intervals in which $(\tilde{x},\tilde{y})$ is
given by~\eqref{Eq:evol_epsilon_y}-\eqref{Eq:evol_epsilon_x}. We first 
assume $w_0=(x_{w_0},y_{w_0})\in M\backslash \overline{(M_0+B(0,\delta))}$. 
In this case, $K_{l}\in \{K_{e},K_e +1\}$.
Without loss of generality, we assume $K_{l}=K_e+1=\frac{K+1}{2}$.
By Lemma~\ref{Lem:estim:time_linearized sol} and because 
$\varepsilon\leq \frac{\varepsilon_\delta}{2}$, the first time 
interval $(t_0,t_1)$ in Definition~\ref{Def:approximate_solution} is bounded 
from above by
\benn
|t_1-t_0| \leq \varepsilon \left(  
\frac{4(1+L_\pm)\mathrm{dist}(w_0, M_0+B(0,\delta))}{f_m(\delta)}
+ \frac{1}{C_{Df}}\right)=:\varepsilon C_0(\delta).
\eenn
We introduce the notation
\beann
&&C_1(\delta):=\max \left\{C_0(\delta),\left( \frac{4(1+L_\pm)\delta}{f_m(\delta)}+ 
\frac{1}{C_{Df}}\right)\right\}, \quad 
C_2(\delta):= C_1(\delta)\left(2C_{D^2}+ 
\frac{\varepsilon_\delta^3}{f_m(\delta)^2}\right),\\
&&C_3(\delta):= 2L C_1(\delta).
\eeann
Lemma~\ref{Lem:estim:time_linearized sol} implies that for 
$i\in\left\{0,\cdots, \frac{K-1}{2}\right\}$ and $t\in [t_{2i},t_{2i+1}]$ 
we may estimate the difference between the full and approximate solutions
by
\benn
| x(t) - \tilde{x}(t)|+ | y(t) - \tilde{y}(t)|
\leq \left[ | x(t_{2i}) - x_{w_{2i}}| + | y(t_{2i}) - y_{w_{2i}}|+ 
\theta^2 C_2(\delta)\right]\txte^{C_3(\delta)}.
\eenn
Lemma~\ref{Lem:estim_time_exact_sol} proves that for 
$i\in\left\{1,\cdots \frac{K-1}{2}\right\}$ and $t\in [t_{2i-1},t_{2i}]$ 
we obtain
\benn
| x(t) - \tilde{x}(t)|+ | y(t) - \tilde{y}(t)|\leq 
\left[ | x(t_{2i-1}) - x_{w_{2i-1}}| + | y(t_{2i-1}) - y_{w_{2i-1}}|
\right]
\txte^{(t_{2i}-t_{2i-1})L \left(1+\frac{1}{\varepsilon}\right)}.
\eenn
This together with $\sum_{i=1}^{\frac{K-1}{2}}(t_{2_i}-t_{2_i-1})\leq T$ 
implies that we can estimate for any $t\in \overline{J_T}$:
\begin{align*}
&| x(t) - \tilde{x}(t)|
	+ | y(t) - \tilde{y}(t)|\\
&\leq \left[ | x(t_{0}) - x_{w_{0}}| + | y(t_{0}) - y_{w_{0}}|+
\frac{K-1}{2}\theta^2 C_2(\delta)\right]
\exp\left(\frac{K-1}{2}C_3(\delta)+TL\left(1+\frac{1}{\varepsilon}\right)
\right)\\
&=\frac{K-1}{2}
\theta^2 C_2(\delta)
\exp\left(\frac{K-1}{2}C_3(\delta)+TL\left(1+\frac{1}{\varepsilon}\right)
\right)=:\theta^2\txte^{\frac{TL}{\varepsilon}}C_4(\delta).
\end{align*}
Analogous estimates apply for $w_0=(x_{w_0},y_{w_0})\in (M_0+B(0,\delta))$.
This proves~\eqref{eq:lemtilde1}. Now the 
results~\eqref{eq:lemtilde2}-\eqref{eq:lemtilde3} follow directly from
the definition of the mapping $\tilde{p}_\varepsilon$ 
in~\eqref{eq:deftildeproject}. 
\end{proof}

Finally we can prove the main result. Some elements of the proof of 
Theorem~\ref{thm:projection} will be kept. However, we can improve the
convergence norm and also simplify the argument that in the singular
limit, solutions satisfy the variational inequality for the generalized 
play operator.

\begin{proof}(of Theorem~\ref{Thm:approimation})
As in the proof of Theorem~\ref{thm:projection}, we shall argue with 
sequences and converging subsequences $\{\varepsilon_k\}$. There exist 
functions $\overline{x},\overline{y}\in \mathrm{W}^{1,q}(J_T)$ such that
\be
p_{\varepsilon_k} \rightarrow \overline{x} \text{ and } y_{\varepsilon_k} 
\rightarrow \overline{y}
\ee
as $\varepsilon_k \rightarrow 0$ weakly in $\mathrm{W}^{1,q}(J_T)$ and 
strongly in $\mathrm{C}(\overline{J_T})$ respectively. Furthermore, one 
shows as previously that $x_{\varepsilon_k}\ra \overline{x}$ in $L^q(J_T)$,
that the convergence of $y_{\varepsilon_k}$ is strong and that $\overline{y}$ 
solves~\eqref{Eq:evol_with_time_x_1}-\eqref{Eq:evol_with_time_x_2} 
with $x=\overline{x}$. The strategy is now to improve the convergence
norm and to simplify the argument that $\overline{x}$ 
solves~\eqref{Eq:evol_with_time_y_1}-\eqref{Eq:evol_with_time_y_3}.

Let $\eta>0$ be arbitrary. Recall that by assumption~\ref{a:Fpm} we must 
have $F_-<F_+$ and both functions are monotone increasing. Using these
facts and the definition~\eqref{eq:deftildeproject} of 
$\tilde{p}_\varepsilon$ yields
\benn
|\tilde{p}_\varepsilon(t) - p_\varepsilon(t)|
\leq  \max\{ |F_-(y_\varepsilon(t)) - F_-(\tilde{y}_\varepsilon(t))|, 
|F_+(y_\varepsilon(t)) - F_+(\tilde{y}_\varepsilon(t))|, 
|x_\varepsilon(t) - \tilde{x}_\varepsilon(t)| \}
\eenn
for all $t\in J_T$. Consider the assumptions and the notation from 
Definition~\ref{Def:approximate_solution} as well as the notation from 
Lemma~\ref{Lem:estim:time_approx_sol}. If $w_0=(x_{w_0},y_{w_0})=(x_0,y_0)\in M\backslash 
\overline{M_0}$ then we set 
\be
\label{eq:deltaeta1}
\delta_\eta:=\min \left\{ \mathrm{dist}(w_0,M_0),
\frac{\eta}{6(1+L_\pm)}\right\} 
\ee
in Definition~\ref{Def:approximate_solution}, so that 
$w_0 \in M\backslash \overline{(M_0+B(0,\delta_\eta))}$, and define 
$\varepsilon_{\delta_\eta}$ as in Definition~\ref{Def:approximate_solution}. 
In this case, we further let
\benn
t_\varepsilon := t_1 \leq \varepsilon C_0(\delta_\eta)
\eenn
for $\varepsilon\in \left(0,\frac{\varepsilon_{\delta_\eta}}{2}\right)$
with $t_1$ and $C_0(\delta_\eta)$ from Lemma~\ref{Lem:estim:time_approx_sol}.
If $w_0=(x_{w_0},y_{w_0})\in \overline{M_0}$ then let
\be
\label{eq:deltaeta2}
\delta_\eta:=\frac{\eta}{6(1+L_\pm)} 
\ee
in Definition~\ref{Def:approximate_solution} 
and again consider $\varepsilon_{\delta_\eta}$ from 
Definition~\ref{Def:approximate_solution}. Then we can define $t_\varepsilon = 0$ 
for $\varepsilon\in \left(0,\frac{\varepsilon_{\delta_\eta}}{2}\right)$. Now we 
can find $\varepsilon_\eta \in \left(0,\frac{\varepsilon_{\delta_\eta}}{2}\right)$ 
such that for all $\varepsilon \in (0,\varepsilon_\eta)$ the following key bound
is satisfied 
\be
\label{eq:epsiloneta}
\max\{2,L_-,L_+\}\theta_\varepsilon^2\txte^{\frac{TL}{\varepsilon}}
C(\delta_\eta)<\frac{\eta}{3}.
\ee
Lemma~\ref{Lem:estim:time_approx_sol} can then be used to estimate for all 
$\varepsilon \in (0,\varepsilon_\eta)$ the fast variable
\bea
\nonumber
\max_{t\in \overline{J_T}: t \geq t_\varepsilon} 
|x_\varepsilon(t) - p_\varepsilon(t)| &\leq& 
\max_{t\in \overline{J_T}: t \geq t_\varepsilon} |x_\varepsilon(t) - 
\tilde{x}_\varepsilon(t)|+ |\tilde{x}_\varepsilon(t) - \tilde{p}_\varepsilon(t)| 
+ |\tilde{p}_\varepsilon(t) - p_\varepsilon(t)|\\
&\leq & 2\delta_\eta(1+L_\pm) + \max\{2,L_\pm\}
\theta_\varepsilon^2\txte^{\frac{TL}{\varepsilon}}
C(\delta_\eta) < \frac{2}{3}\eta.\label{eq:fastesttemp}
\eea
Given a subsequence $\varepsilon_k\ra 0$, we choose $k_\eta>0$ such that 
$\varepsilon_k\in (0,\varepsilon_\eta)$ and
\be
\label{eq:keta}
\max_{t\in \overline{J_T}} |y_{\varepsilon_k}(t) - \overline{y}(t)| 
+ |p_{\varepsilon_k}(t) - \overline{x}(t)| <\frac{\eta}{3} 
\ee
for all $k\geq k_\eta$. From the last two maximum norm 
bounds~\eqref{eq:fastesttemp}-\eqref{eq:keta} it then 
follows that we have 
\bea
\|x_{\varepsilon_k} - \overline{x}\|_{ \mathrm{C}[t_{\varepsilon_k},T] }+
\|y_{\varepsilon_k} - \overline{y}\|_{\mathrm{C}(\overline{J_T})}\nonumber 
	&\leq& \|y_{\varepsilon_k} - \overline{y}\|_{\mathrm{C}(\overline{J_T})} 
	+ \max_{t\in[t_{\varepsilon_k},T]} |x_{\varepsilon_k}(t) 
	- p_{\varepsilon_k}(t)| + |p_{\varepsilon_k}(t) - \overline{x}(t)| \nonumber\\
	&<& \eta \label{estim:Step4}.
\eea
The bound~\eqref{estim:Step4} is the crucial step. It is going to provide
that the variational inequality is solved and it is going to show the 
convergence in the $C^0$-norm. We start by proving the former, i.e., by
showing that $\overline{x}$ 
solves~\eqref{Eq:evol_with_time_y_1}-\eqref{Eq:evol_with_time_y_3} with 
$y=\overline{y}$. The proof of the properties
\benn
\overline{x}(0)=\min\{\max\{x_0, F_-(y_0)\},F_+(y_0)\}
\eenn	
and
\benn
\overline{x}(t)\in [F_-(\overline{y}(t)), F_+(\overline{y}(t))]\qquad  
\forall\ t\in \overline{J_T}
\eenn
remain the same as in the proof of Theorem~\ref{thm:projection}. Let $t_0\in J_T$ 
be given. Suppose that $\overline{x}(t_0)=F_+(\overline{y}(t_0))$. Assume first 
that $g(x,y,t)<0$ in a neighbourhood $U\times I\subset M\times J_T$ of 
$(\overline{x}(t_0),\overline{y}(t_0),t_0)$ and $(\overline{x}(t),\overline{y}(t)) 
\in U$ for all $t \in I$. Using \cite[Theorem 4.15 and Theorem 4.16]{RudinFunc} and 
continuity of $g$, after eventually making $U$ smaller, we may assume that 
\be
\label{eq:Cgreminder}
-C_g<g(x,y,t)<-c<0\qquad \text{for some constant $c$ and all $(x,y,t)\in U\times I$.}
\ee
Now we are going to apply the crucial bound~\eqref{estim:Step4}. Define a parameter 
$\eta>0$ (which will be fixed later) with 
\benn
\eta < \min_{t\in I }\mathrm{dist}((\overline{x}(t),\overline{y}(t)), \partial U). 
\eenn
In dependence of this parameter, let $\delta_\eta,\varepsilon_\eta,k_\eta$ be 
defined as in \eqref{eq:deltaeta1}-\eqref{eq:deltaeta2}, \eqref{eq:epsiloneta} 
and \eqref{eq:keta}. Moreover, we choose $\varepsilon_I\in (0,\varepsilon_\eta)$ 
such that 
\benn
t_\varepsilon \leq \varepsilon C_0(\delta_\eta) < \min\{t: t\in I \} \qquad 
\text{for all $\varepsilon \in (0,\varepsilon_I)$}.
\eenn
Finally we define $k_I \geq k_\eta$ such that $\varepsilon_k \in (0,\varepsilon_I)$ 
for all $k\geq k_I$. These choices then lead us to the estimate
\be
	\|y_{\varepsilon_k} - \overline{y}\|_{\mathrm{C}(\overline{J_T})} 
+ \|x_{\varepsilon_k} - \overline{x}\|_{ \mathrm{C}[t_{\varepsilon_k},T] } 
< \eta\label{estim:Step5}
\ee
for all $k\geq k_I$. Therefore, we obtain 
$(x_{\varepsilon_k}(t),y_{\varepsilon_k}(t),t)\in U\times I$ for all 
$k\geq k_I$ and $t\in I$. Because $(\overline{x}(t_0),\overline{y}(t_0))\in 
\partial M_0$ by assumption, \eqref{estim:Step5} yields $\mathrm{dist}((x_{\varepsilon_k}(t_0),
y_{\varepsilon_k}(t_0)), \partial M_0)<\eta$ for all $k\geq k_I$.

Recall that $-C_g<g(x,y,t)<-c<0$ for all $(x,y,t)\in U\times I$ by~\eqref{eq:Cgreminder}. 
This fact applied in the fast-slow system~\eqref{Eq:evol_epsilon_y}-\eqref{Eq:evol_epsilon_x} 
implies that for all $k\geq k_I$, $y_{\varepsilon_k}$ is monotone decreasing in $I$ 
with $-C_g < \dot{y}_{\varepsilon_k} < -c$. We also already know from~\eqref{eq:fastesttemp}
that
\benn
\max_{t\in \overline{J_T}: t \geq t_\varepsilon} |x_\varepsilon(t) - p_\varepsilon(t)| 
<\frac{2\eta}{3}\quad \text{for $\varepsilon\in (0,\varepsilon_\eta)$} 
\eenn
and since $(p_\varepsilon,y_\varepsilon)\in M_0$ for all $\varepsilon >0$ by 
definition and since $t_\varepsilon < \max\{ t \in I \}$, this yields 
\benn
\max_{t\in I}\mathrm{dist}((x_{\varepsilon_k}(t),y_{\varepsilon_k}(t)), M_0)
<\frac{2\eta}{3}\quad \text{ for all $k\geq k_I$}. 
\eenn		
This fact can be combined with the observation that there is no fast 
flow inside the critical manifold, i.e., $\dot{x}_{\varepsilon_k}(t)=0$ 
if $(x_{\varepsilon_k}(t),y_{\varepsilon_k}(t))\in M_0$ and with the fact that 
$\mathrm{dist}((x_{\varepsilon_k}(t_0),y_{\varepsilon_k}(t_0)), \partial M_0)<\eta$ 
and $-C_f < \dot{y}_{\varepsilon_k} < -c$ in $I$ for all $k\geq k_I$. In particular,
we may now conclude that
\benn
\max_{t\in I\cap [t_0,T]}\mathrm{dist}((x_{\varepsilon_k}(t),y_{\varepsilon_k}(t)), 
\partial M_0)<\eta\quad \text{ for all $k\geq k_I$}.
\eenn
Since we are still free in our choice of $\eta$ (which then determines $k_I$), 
the last estimate, together with the crucial bound~\eqref{estim:Step5}, proves 
that $(\overline{x}(t),\overline{y}(t))\in \partial M_0$ and therefore 
$\overline{x}(t)=F_+(\overline{y}(t))$ for all $t\in I\cap [t_0,T]$. Moreover, 
it follows that $\overline{y}$ is monotone decreasing in $I$. Since $F_+$ is 
monotone increasing, this implies that $\overline{x}$ is monotone increasing 
in $I \cap [t_0,T]$, so that 
$\dot{\overline{x}}(t)<0$ for a.e. $t\in I\cap [t_0,T]$.
		
The case when $g(x,y,t)>0$ in a neighbourhood of $(\overline{x}(t_0),\overline{y}(t_0),t_0)$ 
leads to a contradiction as we would have moved already inside $M_0$ earlier in
this case. If $t_0\in J_T$ is a time such that $g(x,y,t)\geq 0$ in $(U\times I )\cap 
(M\times [t_0,T])$ for some neighbourhood $U\times I$ of $(\overline{x}(t_0),
\overline{y}(t_0),t_0)$ then $\overline{y}$ is monotone increasing and $\overline{x}$ 
is constant in $I \cap [t_0,T]$ with $\overline{x}(t)=F_+(\overline{y}(t_0))$.
The cases when $\overline{x}(t_0)=F_-(\overline{y}(t_0))$ or $(x(t_0),y(t_0))\in 
\mathrm{int}(M_0)$ are treated in a similar manner.
Therefore, we have shown that $\overline{x}$ 
solves~\eqref{Eq:evol_with_time_y_1}-\eqref{Eq:evol_with_time_y_3} with $y=\overline{y}$.
	
Uniqueness of	$\overline{x}$ and $\overline{y}$ and convergence of the whole sequence 
follow just as in the proof of Theorem~\ref{thm:projection}. More precisely, we may repeat the
steps with any subsequence converging to zero and~\eqref{estim:Step4} then proves~\eqref{convergence_yeps_hyst}.
From~\eqref{convergence_yeps_hyst} we deduce the convergence result of $x_\varepsilon$. 
This yields the convergence result also for $y_\varepsilon$.
\end{proof}

\section{An Application to Forced Oscillations}
\label{sec:appl}

So far, we have studied the singular limit convergence guided by Netushil's 
conjecture of fast-slow systems coupled with hysteresis operators. Our results in
Sections~\ref{sec:main1}-\ref{sec:main2} were fully \emph{rigorous}. However,
it will be of interest to see, how we can practically analyze fast-slow 
systems~\eqref{Eq:evol_epsilon_y}-\eqref{Eq:evol_epsilon_x}. In this section
we provide a \emph{numerical} and \emph{formal} analysis of an important
subclass of~\eqref{Eq:evol_epsilon_y}-\eqref{Eq:evol_epsilon_x}.\medskip

First, note that~\eqref{Eq:evol_epsilon_y}-\eqref{Eq:evol_epsilon_x} can 
be re-written in non-autonomous form as
\be
\label{eq:forcedfs}
\begin{array}{rcl}
\varepsilon \frac{\txtd x}{\txtd \tau} &=& f(x,y),\\
\frac{\txtd y}{\txtd \tau} &=& g(x,y,t),\\
\frac{\txtd t}{\txtd \tau} &=& \omega,
\end{array}
\ee
where we introduce an additional parameter $\omega$.
From a dynamical systems point of view, fast-slow problems of the 
form~\eqref{eq:forcedfs} provide many highly investigated classical examples. 
For example, if we consider $f$ as the classical cubic non-linearity and
assume that $g$ is periodic in $t$, say for concreteness, 
\be
f_{\textnormal{vdP}}(x,y)=y-\frac{x^3}{3}+x,\qquad g(x,y,t)=g(x,y,t+2\pi),
\ee
then~\eqref{eq:forcedfs} has a very prominent representative given by the
forced van der Pol oscillator~\cite{vanderPol2,KuehnBook,Guckenheimer2} usually 
written as
\be
\label{eq:forcedvdP}
\begin{array}{rcl}
\varepsilon \frac{\txtd x}{\txtd \tau} &=& f_{\textnormal{vdP}}(x,y),\\
\frac{\txtd y}{\txtd \tau} &=& a \sin(2\pi\theta )-x,\\
\frac{\txtd \theta}{\txtd \tau} &=& \omega. 
\end{array}
\ee
where $(x,y,\theta)\in\R^2\times S^1$ with $S^1=[0,1]/(0\sim 1)$ is the 
circle, $(y,\theta)$ are the slow variables and $a,\omega$ are the main 
amplitude and phase parameters used in bifurcation studies 
of~\eqref{eq:forcedvdP}; see~\cite{GuckenheimerHoffmanWeckesser2,Boldetal,
GrasmanNijmeijerVeling,SzmolyanWechselberger}. 
The forced van der Pol equation is also one of the very few ODE models,
where it has been rigorously proven that chaotic oscillations may 
occur~\cite{Haiduc1}. The model has also a strong link to one-dimensional
or almost one-dimensional return maps and chaotic 
dynamics~\cite{GuckenheimerWechselbergerYoung}. The forced van der Pol
equation is still being studied very actively~\cite{Burkeetal1}.\medskip

In our setting of generalized play operators and Netushil's conjecture, we
have considered a different class of fast variable vector fields specified 
by assumptions~\ref{a:fg}-\ref{a:Fpm}. This naturally raises the question,
what actually happens dynamically, if we replace the fast-vector field in
the van der Pol oscillator with one satisfying~\ref{a:fg}-\ref{a:Fpm}. Our 
main example we propose to study is
\be
\label{eq:fnonlinex}
f(x,y)=\left\{
\begin{array}{rl}
y-x-1 ~ &\text{ if $x<y-1$,}\\
y-x+1 ~ &\text{ if $x>y+1$,}\\ 
0 ~&\text{ else,}\\
\end{array}
\right.
\ee 
for the fast variable. Piecewise linear approximations are very
classical in fast-slow systems, e.g., they have been studied in the van der 
Pol context many times~\cite{Levinson} but are still of high current 
interest~\cite{Desrochesetal1,Desrochesetal2}. For the slow variable, we 
propose a linear term and a sinusoidal forcing
\be
\label{eq:gnonlinex}
g(x,y,t)= a \sin(2\pi t)+b x + cy
\ee
for parameters $a,b,c\in\R$. The strategy to start
with the lowest order Taylor expansion is well-known in fast-slow 
systems for the slow variables~\cite{Guckenheimer7,SzmolyanWechselberger1}
and so is starting with the lowest harmonic in many other 
contexts~\cite{Kuramoto}. One checks that~\ref{a:fg}-\ref{a:gbound} hold 
with $F_+(y)=y+1$ and $F_-(y)=y-1$ except for the boundedness of $\cG(x)$;
however, we shall observe that $x$ is going to stay bounded for certain
parameter choices to be investigated below so we can just cut off $\cG(x)=bx$
smoothly outside a compact set.\medskip

\begin{figure}[htbp]
	\centering
\begin{overpic}[width=1.0\linewidth]{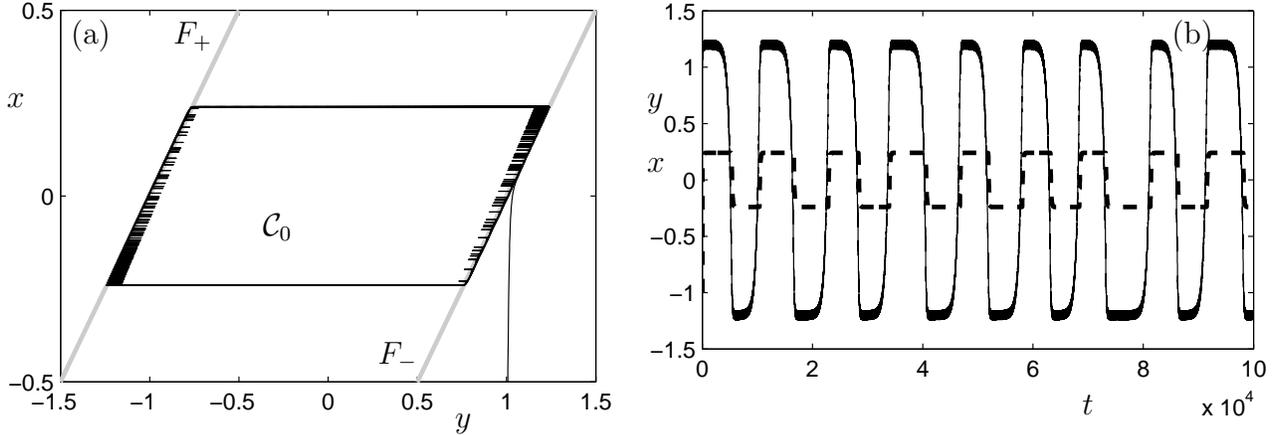}
\put(5,30){(a)}
\put(91,30){(b)}
\put(0,25){$x$}
\put(35,0){$y$}
\put(20,15){$\cC_0$}
\put(13,30){$F_+$}
\put(29,5){$F_-$}
\put(84,1){$t$}
\put(50,25){$y$}
\put(50,20){$x$}
\end{overpic}
\caption{\label{fig:05}Direct numerical integration of the fast-slow ODEs~\eqref{eq:forcedfs} 
with nonlinearities given by~\eqref{eq:fnonlinex}-\eqref{eq:gnonlinex}; the parameters
are chosen as $a=1$, $b=-1$, $c=\frac15$, $\omega=4$, and $\varepsilon=0.01$. The initial
condition was chosen outside and $\cO(1)$-separated from $\cC_0$. (a) Phase portrait showing 
a typical trajectory (black curve) projected into the $(y,x)$-plane, i.e., not 
explicitly showing the non-autonomous periodic $\tau$-direction. $\cC_0$ lies between the 
two curves (gray) defined by $F_\pm$ and $\cC_0$ has the same dimension as the ambient phase.
(b) Time series of the trajectory with the $y$-coordinate (solid curve) and the $x$-coordinate 
(dashed curve). One clearly observes small scale behaviour in the region, where the fast 
and slow variables interact.}
\end{figure}

As a first step, we would like to check, whether we can find any interesting dynamics
by selecting the basic nonlinearities~\eqref{eq:fnonlinex}-\eqref{eq:gnonlinex}. 
Figure~\ref{fig:05} shows the results of numerical integration. We have selected
an initial condition far separated of the critical manifold
\be
\cC_0=\{(x,y,\tau)\in\R^2\times S^1:y-1\leq x\leq y+1\}.
\ee 
The initial condition gets attracted very fast towards $\cC_0$ as shown in 
Figure~\ref{fig:03}(a) as expected already from the theoretical results based 
upon assumption~\ref{a:cm}. The dynamics near the boundary
\be
\partial \cC_0 = \{x=F_-(y)\}\cup \{x=F_+(y)\}=:\cC_-\cup \cC_+
\ee
is a lot more delicate. We observe in Figure~\ref{fig:05} the case of many small 
amplitude oscillations (SAOs) near both parts of the boundary. Furthermore, there 
are relatively slow jumps between $\cC_-$ and $\cC_+$ in comparison to the long 
very slow drift time near each boundary piece. Essentially, we observe oscillations,
which look similar to classical relaxation oscillations~\cite{MisRoz,Grasman,KuehnBook},
just with high-frequency fast SAOs overlayed near the slowest scale pieces and the 
jumps in the relaxation cycle still occur on the slow time scale. Figure~\ref{fig:05}
does not provide an indication, whether the oscillations are actually periodic or
potentially even chaotic.\medskip

The SAOs near $F_\pm$ are easy to explain formally. Suppose we use the 
standard slow subsystem reduction just along the lines $\cC_\pm$, then 
we obtain
\be
\label{eq:1DODE}
\frac{\txtd y}{\txtd t} = a \sin(2\pi \omega t)+b x + cy=
a \sin(2\pi \omega t)+(b+c)y \pm b,\qquad y(0)=y_0.
\ee

\begin{figure}[htbp]
	\centering
\begin{overpic}[width=1.0\linewidth]{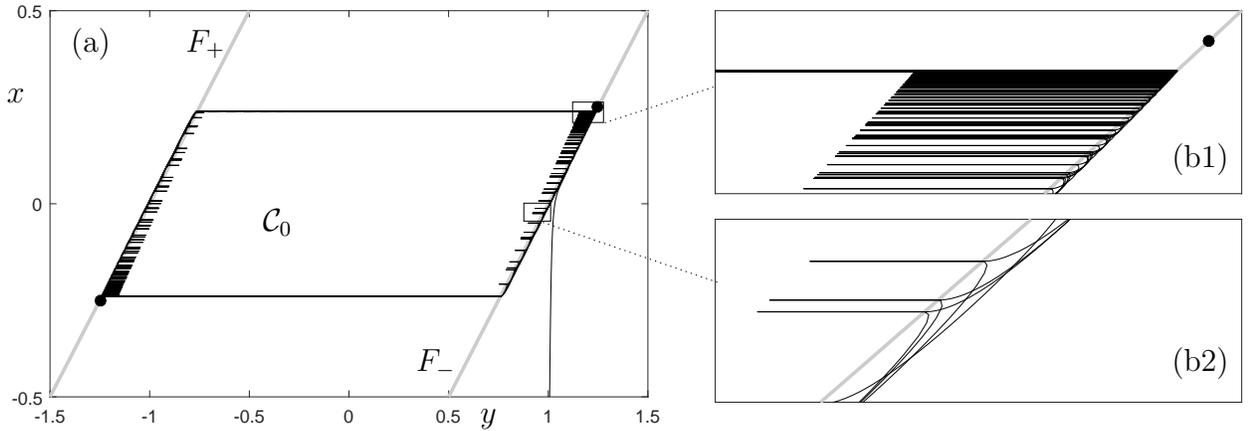}
\put(5,29){(a)}
\put(91,20){(b1)}
\put(91,4){(b2)}
\put(0,25){$x$}
\put(37,0){$y$}
\put(20,15){$\cC_0$}
\put(14,29){$F_+$}
\put(32,4){$F_-$}
\end{overpic}
\caption{\label{fig:06}Direct numerical integration of the fast-slow ODEs~\eqref{eq:forcedfs} 
with nonlinearities given by~\eqref{eq:fnonlinex}-\eqref{eq:gnonlinex}; the parameters
are chosen as $a=1$, $b=-1$, $c=\frac15$, $\omega=4$, $\varepsilon=0.01$ and final 
time $T=5\cdot 10^4$. The initial condition was chosen outside and $\cO(1)$-separated 
from $\cC_0$. (a) Phase portrait showing a typical trajectory (black curve) projected 
into the $(y,x)$-plane. We have now also marked the two ``average'' equilibrium 
points $(Y_\pm,F_\pm(Y_\pm))$ as dots (black). (b1) Zoom near $(Y_-,F_-(Y_-))=
(1.25,0.25)$. (b2) Zoom near a typical region with dynamics well-approximated by 
the formal slow subsystem~\eqref{eq:1DODE} near the branch $\cC_-$.}
\end{figure}

The ODE~\eqref{eq:1DODE} can actually be solved explicitly. We denote the 
solutions corresponding to the respective signs in front of the constant term 
$\pm b$ by $y_\pm=y_\pm(t)$. We have
\be
\label{eq:ysolbase}
y_\pm(t)=Y_\pm+Y_\txte \txte^{(b+c)t}+Y_\txth(t)
\ee
where the individual (``constant, exponential prefactor, and harmonic'') terms 
are given by
\bea
Y_\pm&=& \mp \frac{b}{b+c}, \label{eq:ysol1}\\
Y_\txte&=& \pm \frac{b}{b+c}+\frac{2 \pi  a \omega }{(b+c)^2+4 \pi ^2 \omega ^2}+y_0, 
\label{eq:ysol2}\\
Y_\txth(t)&=& -\frac{a (b+c) \sin (2 \pi  t \omega )+2 \pi  
a \omega  \cos (2 \pi  t \omega )}{(b+c)^2+4 \pi ^2 \omega ^2}. \label{eq:ysol3}
\eea

\begin{figure}[htbp]
	\centering
\begin{overpic}[width=1.0\linewidth]{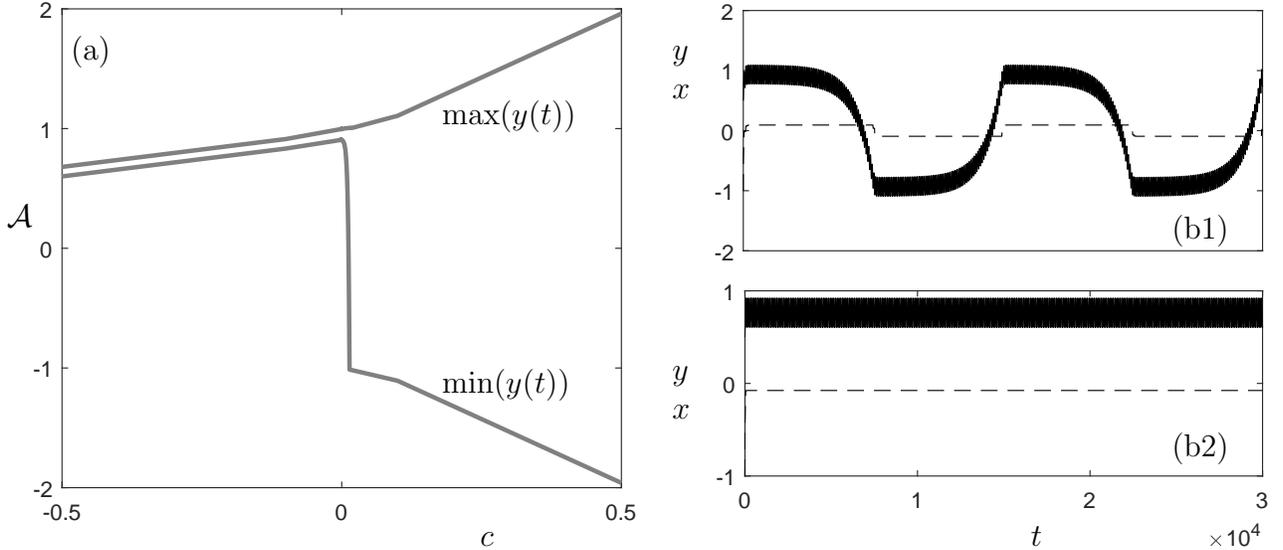}
\put(5,38){(a)}
\put(91,24){(b1)}
\put(91,7){(b2)}
\put(0,25){$\cA$}
\put(37,0){$c$}
\put(80,0){$t$}
\put(52,10){$x$}
\put(52,13){$y$}
\put(52,35){$x$}
\put(52,38){$y$}
\put(34,33){$\max(y(t))$}
\put(34,12){$\min(y(t))$}
\end{overpic}
\caption{\label{fig:07}Bifurcation diagram of the fast-slow ODEs~\eqref{eq:forcedfs} 
with nonlinearities given by~\eqref{eq:fnonlinex}-\eqref{eq:gnonlinex}; the parameters
are chosen as $a=1$, $b=-1$, $\omega=4$, and $\varepsilon=0.01$. (a) Main bifurcation
diagram varying the parameter $c$ and showing the maximum and minimum amplitudes $\cA$
of the variable $y$ for the global attractor. (b1) Time series for $c=0.1$. (b2) Time 
series for $c=-0.1$. The dashed curves are $x(t)$ and the solid curve with SAOs are 
$y(t)$.}
\end{figure}

The formal slow subsystem~\eqref{eq:1DODE} remains bounded for all $y_0\in\R$ 
if and only if $b+c\leq 0$. We shall not investigate the borderline case $b=-c$
here and just assume $b+c<0$ from now on. Then 
$Y_\txte \exp[(b+c)t]\ra 0$ as $t\ra +\I$ so the dynamics of $y_\pm(t)$ is a 
harmonic oscillation around the points $Y_\pm$, i.e., we have
\benn
\lim_{t\ra +\I}\frac{1}{t}\int_0^t y_\pm(s)~\txtd s = Y_\pm
\eenn
so we may view $Y_\pm$ as averaged equilibrium points. We now have to re-visit the
numerical results from Figure~\ref{fig:05}, which are presented in phase space
in a different way in Figure~\ref{fig:06}, where we clearly see that the slow 
subsystem approximation correctly describes the behaviour near the branches $\cC_\pm$,
i.e., we move upwards via the terms $Y_-+\txte^{(b+c)t}Y_\txte$ on the right branch 
$\cC_-$ towards $(y,x)=(Y_-,F_-(Y_-))$ and there are oscillations induced by the term 
$Y_\txth(t)$. Similarly we move downwards on the left branch $\cC_+$ with several
oscillations induced by the time-dependent terms.\medskip

\begin{figure}[htbp]
	\centering
\begin{overpic}[width=1.05\linewidth]{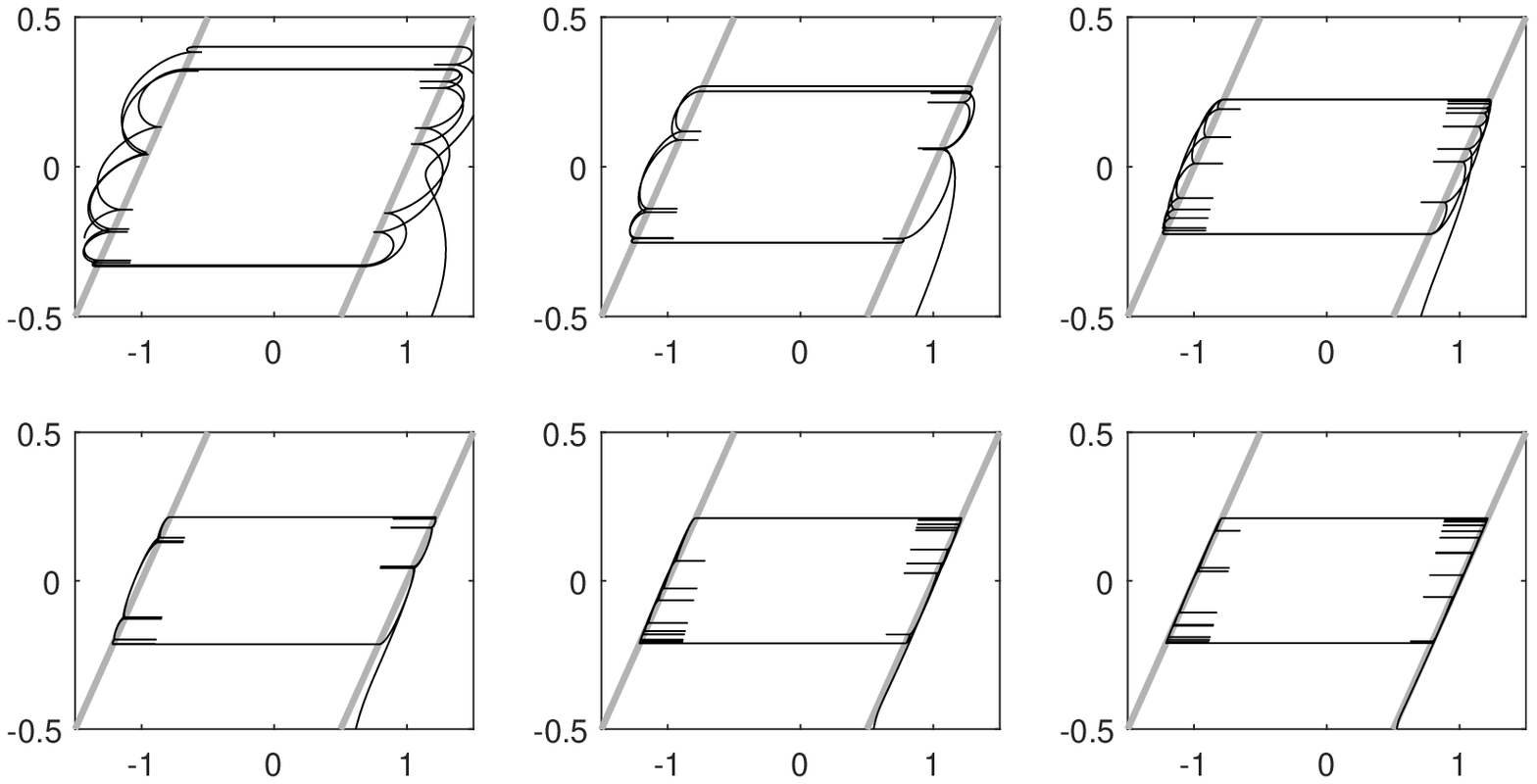}
\put(5,40){$x$}
\put(90,0){$y$}
\end{overpic}
\caption{\label{fig:08}Numerical simulation of the fast-slow ODEs~\eqref{eq:forcedfs} 
with nonlinearities given by~\eqref{eq:fnonlinex}-\eqref{eq:gnonlinex}; the parameters
are chosen as $a=1$, $b=-1$, $c=\frac15$, $\omega=4$. The parameter $\varepsilon$ is 
varied. The top row shows $\varepsilon=0.5,0.2,0.1$ and the bottom 
row $\varepsilon=0.05,0.02,0.01$.}
\end{figure}

The next natural question is, how the global periodic large oscillations are generated
under parameter variation. Figure~\ref{fig:07} shows a basic bifurcation diagram fixing
all parameters except $c$. We observe a very rapid growth of the amplitude of the oscillations
as $c$ passes through $c=0$. In particular, the transition could be viewed as being similar
to a canard-type explosion~\cite{DienerDiener,Eckhaus,DumortierRoussarie,KuehnBook} as the 
growth of the amplitude occurs near a part of $\cC_0$, which
is not normally hyperbolic and not attracting, i.e., inside $\textnormal{int}(\cC_0)$. Note 
carefully that if $x\approx 0$, then the slow equation for $y$ has only a small $x$-dependence, 
so $c$ can actually control growth or decay in this region. Indeed, in this case the singular
limit generalized play operator from Theorem~\ref{thm:projection} precisely shows an equilibrium 
point at $x=0$ for the $y$-dynamics if $c<0$.\medskip  

The last step we would like to check is to illustrate numerically the convergence of 
the fast-slow system to the system coupled with a generalized play operator depending upon $\varepsilon$.
Figure~\ref{fig:08} shows, how we converge from an oscillation with quite large 
excursions outside of $\cC_0$ to the singular limit generalized play operator, which
is entirely constrained to $\cC_0$ after the projection of the initial condition. We
observe that on the initial transient approach towards the oscillatory solution, there
are significant differences in the patterns of the SAOs for different values of 
$\varepsilon$. Furthermore, the patterns seem to stabilize a bit more as 
$\varepsilon\ra 0$ with more oscillations near the averaged equilibrium points
discussed above. This suggests that an asymptotic description of the precise patterns
could be possible locally but we leave this as an aspect for future work. Similarly,
one could consider a more detailed parameter study, which should also be considered
in another context focusing more on several classes of models. Here we only wanted
to illustrate the proof of Netushil's conjecture in our setting of coupled fast-slow systems
and show that the associated systems can have interesting nontrivial dynamics. \medskip

\textbf{Acknowledgments:} CK has been supported by a Lichtenberg Professorship of the 
VolkswagenStiftung. CM has been supported by the DFG through the International Research 
Training Group IGDK 1754 ``Optimization and Numerical Analysis for Partial Differential 
Equations with Nonsmooth Structures''. We also would like to thank two anonymous referees
for suggestions regarding the presentation of our results.

\bibliographystyle{alpha}
\bibliography{../my_refs}

\end{document}